\documentclass[a4paper, 10pt]{article}

\usepackage[a4paper, total={6in, 9.5in}]{geometry}
\usepackage[utf8]{inputenc}
\usepackage{amssymb}
\usepackage{amsthm}
\usepackage{amsmath}
\usepackage{amsfonts}

\usepackage{graphicx}

\usepackage{titlesec}

\usepackage{comment}
\usepackage{titling}
\usepackage{tikz-cd}

\usepackage{enumitem}

\usepackage{graphicx}
\usepackage{hyperref}
\usepackage{mathrsfs}
\usepackage{mathtools}

\usepackage[mathscr]{euscript}

\usepackage{stmaryrd}

\titlelabel{\thetitle.\quad}
\titleformat{\subsubsection}[runin]{\normalfont\bfseries}{\thesubsubsection.}{3pt}{}
\newcounter{cislo} \numberwithin{cislo}{section}

\numberwithin{equation}{section}

\usepackage{mathabx,epsfig}

\usepackage{xparse}
\usepackage{etoolbox}
\usepackage{letltxmacro}
\makeatletter
\AtBeginDocument{%
\def\strip@@parentheses(#1){#1}
\LetLtxMacro\enumerate@@item\item
\AtBeginEnvironment{enumerate}{%
  \RenewDocumentCommand{\item}{o}{%
    \IfValueTF{#1}{
      \enumerate@@item[#1]%
      \protected@edef\@currentlabel{\strip@@parentheses#1}
    }{%
      \enumerate@@item
    }%
  }%
}
}
\makeatother

\newtheorem{theorem}[cislo]{Theorem}
\newtheorem{lemma}[cislo]{Lemma}
\newtheorem{proposition}[cislo]{Proposition}
\newtheorem{corollary}[cislo]{Corollary}

\newtheorem{observation}[cislo]{Observation}

\newtheorem*{maintheorem}{Theorem}

\theoremstyle{definition}
\newtheorem{definition}[cislo]{Definition}
\newtheorem{notation}[cislo]{Notation}
\newtheorem{remark}[cislo]{Remark}

\newtheorem{example}[cislo]{Example}

\newtheorem{discussion}[cislo]{Discussion}

\newtheorem{construction}[cislo]{Construction}
\newtheorem{setup}[cislo]{Setup}

\newtheorem{recollection}[cislo]{Recollection}

\theoremstyle{remark}

\let\top\relax
\newcommand{\top}{\mathrm{top}}
\newcommand{\st}{\mathrm{st}}

\newcommand{\llp}{(\!(}
\newcommand{\rrp}{)\!)}
\newcommand{\llb}{\llbracket}
\newcommand{\rrb}{\rrbracket}

\newcommand{\cn}{\mathrm{cn}}

\newcommand{\ans}{\mathrm{ans}}
\newcommand{\qp}{\mathrm{qp}}

\DeclareMathOperator{\Sch}{Sch}

\DeclareMathOperator{\RG}{R\Gamma}

\DeclareMathOperator{\Rep}{Rep}

\DeclareMathOperator{\Sp}{Sp}
\DeclareMathOperator{\CycSp}{CycSp}

\DeclareMathOperator{\mo}{mod}

\let\lim\relax
\DeclareMathOperator*{\lim}{lim}

\DeclareMathOperator{\coker}{coker}

\DeclareMathOperator{\Qcoh}{QCoh}
\DeclareMathOperator{\Vect}{Vect}
\DeclareMathOperator{\Irr}{Irr}

\DeclareMathOperator{\fib}{fib}

\newcommand{\eB}{\mathscr{B}}

\newcommand{\eV}{\mathscr{V}}

\newcommand{\cl}{\mathrm{cl}}

\let\inf\relax
\newcommand{\inf}{\mathrm{inf}}
\newcommand{\cdh}{\mathrm{cdh}}

\newcommand{\HCn}{HC^{-}}
\newcommand{\HC}{HC}
\newcommand{\LHH}{L_{\mathrm{cdh}}HH}
\newcommand{\LHCn}{L_{\mathrm{cdh}}HC^{-}}
\newcommand{\LHC}{L_{\mathrm{cdh}}HC}

\DeclareMathOperator{\id}{id}

\DeclareMathOperator{\Spec}{Spec}
\DeclareMathOperator{\red}{red}
\DeclareMathOperator{\chara}{char}

\DeclareMathOperator{\A}{\mathbb{A}}
\let\P\relax
\DeclareMathOperator{\P}{\mathbb{P}}
\DeclareMathOperator{\Gr}{\mathbf{Gr}}

\DeclareMathOperator{\dStk}{dStk}

\DeclareMathOperator{\Cat}{Cat}

\DeclareMathOperator{\Flag}{Flag}
\DeclareMathOperator{\Grass}{Grass}

\DeclareMathOperator{\Fix}{Fix}

\newcommand{\Gm}{\mathbb{G}_m}

\newcommand{\GL}{GL}

\newcommand{\lop}{\mathrm{L^+}}

\newcommand{\fl}{\mathrm{fl}}

\let\O\relax
\DeclareMathOperator{\O}{\mathscr{O}}
\DeclareMathOperator{\eL}{\mathscr{L}}
\DeclareMathOperator{\eE}{\mathscr{E}}
\DeclareMathOperator{\eF}{\mathscr{F}}

\DeclareMathOperator{\eC}{\mathscr{C}}
\DeclareMathOperator{\eD}{\mathscr{D}}
\DeclareMathOperator{\eH}{\mathscr{H}}
\DeclareMathOperator{\eX}{\mathscr{eX}}

\DeclareMathOperator{\F}{\mathbb{F}}
\DeclareMathOperator{\Z}{\mathbb{Z}}
\DeclareMathOperator{\Q}{\mathbb{Q}}
\DeclareMathOperator{\N}{\mathbb{N}}
\DeclareMathOperator{\C}{\mathbb{C}}

\DeclareMathOperator{\rank}{rank}

\DeclareMathOperator{\pt}{pt}

\newcommand{\perf}{\mathrm{perf}}

\DeclareMathOperator{\Perf}{Perf}

\usepackage[style=alphabetic, maxnames = 5, maxalphanames=5]{biblatex}
\usepackage[style=alphabetic]{biblatex}
\renewbibmacro{in:}{}

\title{Equivariant localizing invariants of simple varieties}
\author{Jakub Löwit}
\date{}

\begin{document}

\maketitle
\begin{abstract}
We define a certain class of {\it simple varieties} over a field $k$ by a constructive recipe and show how to control their (equivariant) truncating invariants. Consequently, we prove that on simple varieties:
\begin{itemize}
    \setlength\itemsep{0em}
    \item if $k=\overline{k}$ and $\chara k = p$, the $p$-adic cyclotomic trace is an equivalence,
    \item if $k = \Q$, the Goodwillie--Jones trace is an isomorphism in degree zero,
    \item we can control homotopy invariant $K$-theory $KH$, which is equivariantly formal and determined by its topological counterparts.
\end{itemize}
\noindent Simple varieties are quite special, but encompass important singular examples appearing in geometric representation theory. We in particular show that both finite and affine Schubert varieties for $GL_n$ lie in this class, so all the above results hold for them.
\end{abstract}

\tableofcontents

\section{Introduction}
\subsection{Motivation}
The aim of this note is to record a technique for comparing and computing equivariant localizing invariants on a simple class of (often singular) varieties arising in geometric representation theory. 

Let $k$ be a base field and $G$ be an algebraic group over it. We define classes $\eB^G_k$ and $\eC^G_k$ of certain proper $G$-equivariant $k$-schemes by a simple constructive recipe and call them {\it simple varieties}. (For the main results, we will assume that $G$ is nice -- for example an $n$-dimensional split torus $T = \Gm^n$. More generally, any linearly reductive group $G$ qualifies.)

%
Although such varieties are rather special, we nevertheless recover singular varieties of interest in geometric representation theory: {\it finite} and {\it affine Schubert varieties} for $\GL_n$ are the most prominent examples.
In fact, these varieties are the original source of our motivation. The understanding of their (equivariant) cohomological invariants is an important topic in geometric representation theory; the most delicate information is controlled by their singularities. For instance, their intersection cohomology governs the representation theory of the algebraic group $\GL_n$ through the geometric Satake equivalence; this has further equivariant extensions. Such phenomena are expected to generalize to more complicated topological invariants.

One source of interesting invariants appearing in algebraic geometry are the {\it localizing invariants}. The most important examples come either from $K$-theory -- such as algebraic $K$-theory $K(-)$ itself, its homotopy invariant version $KH(-)$, or its topological realization $K_{\top}(-)$ over $\C$; or from Hochschild homology -- such as Hochschild homology $HH(-/k)$, negative cyclic homology $HC^{-}(-/k)$, topological cyclic homology $TC(-)$, and other variants. Localizing invariants can be evaluated on any derived stack through its category of perfect complexes. In particular, we can evaluate them on a global quotient $X/G$ of any $G$-equivariant variety $X$, obtaining their {\it equivariant} versions.

\subsection{Main results}
The starting point of this project is to understand equivariant $K$-theory of affine Schubert varieties and trace maps from it, with a view towards geometric representation theory. In \cite{Low24} we described the relationship between $T$-equivariant $K$-theory and $T$-equivariant Hochschild homology of affine Schubert varieties after perfection in characteristic $p$. It turns out that our technique can be pushed further to control any truncating invariant (without perfection, in any characteristic). 

This has two kinds of applications. On one hand, we can concretely control and compute such truncating invariants, most notably $KH(-)$. These are in particular equivariantly formal in a rather strong sense -- while such formality statements were expected, they were missing until now. On the other hand, we can establish isomorphisms between different localizing invariants (without really computing them) as long as their difference is truncating -- most notably between $K$-theory and variants of topological cyclic homology.

An abridged and simplified version of our present results goes as follows.
\begin{maintheorem}
Let $k$ be a field and $G$ a nice  (or linearly reductive) group over it\footnote{See Remark \ref{remark: conditionally G lineraly reductive} and refereces there for discussion of these assumptions.}. We construct classes $\eB^G_k \subseteq \eC^G_k$ of $G$-equivariant projective $k$-schemes, which for example contain:
\begin{itemize}
    \item projective spaces, Grassmannians, partial flag varieties,
    \item classical Schubert varieties in Grassmannians,
    \item affine Schubert varieties in the $GL_n$ affine Grassmannian.
\end{itemize}
For such varieties, the following statements hold:
\begin{itemize}
    \item If $k$ is an algebraically closed field of characteristic $p$ and $X \in \eC^G_k$, then the mod-$p$ cyclotomic trace induces an equivalence of ring spectra
    \begin{equation*}
        K^G(X; \F_p) \xrightarrow{\simeq} TC^G(X; \F_p).
    \end{equation*}
    Similarly for the $p$-adic version.
    \item If $k = \Q$ and $X \in \eB^G_{\Q}$, then the Goodwillie--Jones trace induces a ring isomorphism
    \begin{equation*}
        K^G_0(X; \Q) \xrightarrow{\cong} (HC^{-})^G_0(X/ \Q).    
    \end{equation*}
    We further get a decomposition $K^G_i(X) \xrightarrow{\cong} KH^G_i(X) \oplus (HC^{-})^G_i(X/ \Q)$ for all $i \geq 1$.
    \item For any $k$ and $X \in \eB^G_k$, there is a natural ring isomorphism
    \begin{equation*}
        KH^G_0(X) \underset{K_0(\pt)}{\otimes} K_{\bullet}(\pt) \xrightarrow{\cong}  KH^G_{\bullet}(X)    
    \end{equation*}
    and $KH^G_0(X)$ is finite projective over $K^G_0(\pt)$. 
    In other words, the ring spectrum $KH^G(X)$ is connective and finite flat over $K^G(\pt)$. When $k=\C$, we can further compare to topological $K$-theory: $KH^G_0(X) \cong K^G_{\top, 0}(X)$.
\end{itemize}
\end{maintheorem}
\begin{proof}
See Example \ref{example: partial flag varieties}, Lemma \ref{lemma: finite schubert varieties are simple} and Lemma \ref{example: affine schubert varieties are simple} for proofs that the mentioned varieties are simple. The postulated equivalences are then respectively given by Theorem \ref{corollary: K and TC}, Theorem \ref{theorem: rational goodwillie--jones trace in degree zero} (with Proposition \ref{proposition: direct sum in positive degrees}) and Theorem \ref{corollary: KH of simple varieties} (with Proposition \ref{theorem: isomorphisms with topological theories in degree zero}). The variant for linearly reductive $G$ is discussed in Appendix \ref{appendix: actions of linearly reductive G}.
\end{proof}

\begin{remark}[Assumptions on $G$]\label{remark: conditionally G lineraly reductive}
The above theorem works cleanest for \textit{nice} groups $G$ over $k$, see \S \ref{section: an interlude on groups and representations}. This covers algebraic tori $T = \Gm^{n}$, groups of roots of unity $\mu_{\ell}$ and finite groups of order coprime to $\chara k$ (as well as their extensions, subquotients and forms). The technical restriction boils down to \cite[Corollaries 5.2.3 and 5.2.6]{ES21} and its variants, which are applicable to quotients by nice groups only. 

However, the machinery of \cite{LS25} together with \cite{KR18, BKRS22, Kha18} allows to extend the aforementioned results to quotients by \textit{linearly reductive groups}; we outline the arguments in Appendix \ref{appendix: actions of linearly reductive G}. With this extension, all our arguments go through without change for any linearly reductive group $G$ over $k$. This does not make a difference if $\chara k = p$ (where the notions of nice and linearly reductive coincide), but considerably extends the applicability if $\chara k = 0$ to any reductive group (such as $G = GL_n$). 
\end{remark}

\begin{remark}[Globality]\label{remark: globality}
Let us highlight the following: our results are global. All the considered schemes are proper over $k$ and our arguments stay within that realm. Even though localizing invariants satisfy Zariski (even Nisnevich) gluing, the desired isomorphisms do not hold Zariski locally. It is interesting that such isomorphisms nevertheless hold globally in naturally occurring singular examples.   
\end{remark}

\subsection{Discussion and context}
Let us emphasize that localizing invariants are highly sensitive to singularities -- to the extent that their values are known only in the simplest singular cases. 
Even if the varieties in question have affine pavings (i.e. are cellular), there is no obvious method for the computations and comparisons of the localizing invariants we care about. The only way of attack is by detailed understanding of resolutions; this can be quite challenging. In effect, not many computations of equivariant $K$-theory nor $KH$ in singular situations are known. We hope to contribute towards filling in this gap.

\subsubsection{Discussion of results.}
Let us discuss the three comparisons from our main theorem in more detail. We emphasize that all these comparisons are compatible with the ring structures on the invariants at hand. 
\begin{itemize}
    \item The $p$-adic cyclotomic trace \S \ref{section: the p-adic cyclotomic trace}.
    \newline The cyclotomic trace $K(-) \to TC(-)$ is the closest known approximation of algebraic $K$-theory by differential forms. In characteristic $p$ with mod-$p$ coefficients, this approximation is rather good. Nevertheless, it is known to be an equivalence only in specific local examples -- see Remark \ref{remark: non-equivariant case is interesting} -- and it is not an equivalence in general. Theorem \ref{corollary: K and TC} produces interesting (equivariant) singular projective examples where this is the case: for any $\eC^G_k$, the map
    \begin{equation*}
        K^G(X; \F_p) \xrightarrow{\simeq} TC^G(X; \F_p)
    \end{equation*}    
    is still an equivalence -- we are not aware of other nontrivial results of this form.

    While the above generality is completely sufficient to cover the examples of our interest, the theorem easily generalizes to a slightly bigger class $\eC_p$ of algebraic stacks, including some mixed-characteristic situations. We record this into a separate discussion in Theorem \ref{theorem: p-adic cyclotomic trace of Cp}.
    
    \item Goodwillie--Jones trace \S \ref{section: rational goodwillie--jones trace in degree zero}.
    \newline When working over $\Q$, the cyclotomic trace specializes to the historically older Goodwillie--Jones trace $K(-) \to HC^{-}(-/\Q)$ and we prove the isomorphisms
    \begin{equation*}
        K^G_0(X; \Q) \xrightarrow{\cong} (HC^{-})^G_0(X/ \Q)    
    \end{equation*}
    for $X \in \eB^G_k$ in Theorem \ref{theorem: rational goodwillie--jones trace in degree zero}. We further get a somewhat counterintuitive decomposition 
    \begin{equation*}
        K^G_i(X) \xrightarrow{\cong} KH^G_i(X) \oplus (HC^{-})^G_i(X/ \Q), \qquad i \geq 1. 
    \end{equation*}
    in Proposition \ref{proposition: direct sum in positive degrees}; this latter argument requires to use the supplementary truncating invariant $\LHCn(-/ \Q)$ reviewed in Appendix \ref{appendix: cdh sheaffified negative cyclic homology}. We illustrate that this decomposition can be nontrivial in Example \ref{example: the projective cone of projective line}.
    
    \item Homotopy invariant $K$-theory \S \ref{section: homotopy invariant K-theory and equivariant formality}. 
    \newline The closest well-behaved approximation of algebraic $K$-theory of singular varieties is its homotopy invariant version $KH$. A reader interested in explicit ring presentations of equivariant invariants may want to focus on this part. Although no computations of $KH$ in our setup appear in the literature, one expects to get reasonable results. 
    This is indeed what we obtain: the natural identification
    \begin{equation*}
        KH^G_0(X) \underset{K_0(\pt)}{\otimes} K_{\bullet}(\pt) \xrightarrow{\cong}  KH^G_{\bullet}(X)    
    \end{equation*}    
    from Theorem \ref{corollary: KH of simple varieties} shows that $KH^G(X)$ is equivariantly formal in a strong sense for any $X \in \eB^G_k$, see also Remark \ref{remark: change of group and equivariant formality}. More conceptually, we prove that the ring spectrum $KH^G(X)$ is connective and finite flat over $K^G(\pt)$. Equivariant formality statements of this form for $KH$ are missing at the moment; we expect them to be useful for computations by Remark \ref{remark: equivariant formality of KH is useful for computations}. We emphasize that $KH^G(X)$ is often explicitly computable -- we do not seriously discuss this here, but see \cite[Remark 6.3, Examples 6.1 and 6.2]{Low24} for some formulas.

    The above formality has some immediate elementary consequences. 
    \newline Firstly, for simple varieties from $\eB^G_{\C}$, we deduce an isomorphism
    \begin{equation*}
        KH^G_0(X) \cong K^G_{\top, 0}(X)
    \end{equation*}
    with topological $K$-theory of the analytification in Proposition \ref{theorem: isomorphisms with topological theories in degree zero}. A posteriori, in these cases $KH$ glues from affine pavings (since topological $K$-theory does). The whole $KH^G_{\bullet}(X)$ is determined by the independent geometric contribution of $K^G_{\top, 0}(X)$ and arithmetic contribution of $K_{\bullet}(\C)$.
    \newline Secondly, simple varieties from $\eB^G_{\F_q}$ satisfy Parshin's property by Example \ref{example: parshins property for simple varieties}.
\end{itemize}

\subsubsection{Strategy of proofs.}
In all three applications, our proof works as follows. We first express the desired statement in terms of some truncating invariant: in the first two cases, the fiber of the comparison map is truncating; in the third case $KH$ is truncating to start with. We then check the statement on the point and consequently on the equivariant point $BG$; this can be usually done ``by hand".

Once the above is settled, we are in a good shape. Indeed, we can control truncating invariants throughout the constructive steps defining simple varieties, so the results reduce to the case of $BG$. The subtle difference between the classes $\eB^G_k$ and $\eC^G_k$ morally depends on whether we wish to control the outputted homotopy groups one by one (as in the latter two applications), or if we are able to get a statement on the level of spectra (as in the first application).

To deduce our results for specific varieties (such as finite or affine Schubert varieties for $\GL_n$), we just need to argue that they are simple. These geometric arguments are completely independent from the homotopical machinery of localizing invariants. 

We would like to point out that the presentation of the proof we are offering goes backwards -- the ideas originate from abstracting the case of affine Schubert varieties.

\subsubsection{Outlook.}\label{section: outlook}
The above results should not be viewed as a definitive list of applications; it should rather give a taste of what is possible. 

\bigskip

In one direction, one could try to come up with other naturally occurring varieties which happen to be simple. Such endeavors can stay entirely withing the realm of classical algebraic geometry via Definition \ref{definition: closure properties on schemes}. 

However, one can further optimize the definition. To start with, one can directly add varieties $X$ such that $\Perf(X)$ has a \textit{full exceptional collection} to our class $\eB_k$ -- this immediately imports further examples where our results hold. For example, this applies to partial flag varieties of arbitrary type \cite{SK24}. Also split smooth projective toric varieties qualify \cite{Kaw05} (and one can generate singular ones via cdh descend in our bigger class $\eC_k$). One may also cleanly formulate our constructive arguments on rather general derived stacks.

We do think that there are further examples coming from geometric representation theory where our results hold. Specifically, it is reasonable to expect that our methods may be adapted to cover finite and affine Schubert varieties of arbitrary types. 

\bigskip

In another direction, one could try to carry through more refined computations and comparisons of localizing invariants along the same recipe. For example, following the constructive definition of simple varieties, $KH$ (as well as other truncation invariants) are actually computable. 

Furthermore, all the comparison maps are compatible with motivic filtrations and our results should refine to this level -- for instance, our results on $KH$ should refine to the level of cdh-local motivic cohomology $\Z(j)^{\cdh}$ of \cite{BEM25}. However, some technical care would be needed, as motivic cohomology complexes do not apriori exist as localizing invariants. 

\subsubsection{Other work.}\label{section: other work}
To close the introduction, we would like to give an idea of the complexity of $K$-theory of the singular varieties we are dealing with.

\begin{remark}
Consider the (affine) conic singularity $\{ z^2 - xy = 0\} \subseteq \A^3_k$ over a field $k$ of characteristic zero. The rather nontrivial computation of its (non-equivariant) $K$-theory is carried out in \cite[Theorem 4.3]{CHWW13}; also see \cite[\S 2.3, \S 2.4]{PS21} for complementary results on finite quotient singularities.

The projective conic coincidentally appears as the smallest singular affine Schubert variety for $\GL_2$ -- it lies in the class $\eB_k$ where all our results apply. See Example \ref{example: the projective cone of projective line}.
\end{remark}

\begin{remark}\label{remark: relation to motivic satake}
There is a notable amount of recent activity on equivariant $K$-theoretic and motivic study of finite and affine flag varieties, with serious applications in geometric representation theory -- \cite{SW18, SVW18, EK19, Ebe22, ES22, ES23, Ebe24, Ebe24b, EE24, EE25}, to name a few. In particular, the formalism of \textit{reduced $K$-motives}, developed and successfully employed in \cite{ES23, Ebe24b, EE24, EE25}, provides a categorified variant of equivariant algebraic $K$-theory. However, in order to control equivariant formality, this variant enforces (by hand) the higher $K$-groups of the base to vanish.

As far as we know, equivariant formality statements such as our Theorem \ref{corollary: KH of simple varieties} were not available before. This result suggests that some equivariant formality holds directly, without modifying the definitions (and including higher $K$-theory). We hope that our methods could help to clarify such issues. 
\end{remark}

\begin{remark}\label{remark: cellular arguments}
Let us also emphasize the difference between our approach and standard gluing techniques for smooth cellular varieties. 

Let $G$ be a linearly reductive group over $k$; assume that $X \in \Sch^G_k$ is smooth and posseses an equivariant paving by affine spaces $\A^n$. By smoothness, $K$-theory, $KH$ and $G$-theory agree; their common value on equivariant $\A^n$ then matches its value on the equivariant point, which is explicit by \S \ref{section: representations and decomposability}. Furthermore, these invariants have the closed-open localization sequence, inherited from $G$-theory. A skeletal induction then allows to control $K^G(X)$ at least as a module -- indeed, one may inductively argue that the long exact sequence associated to a gluing of an open cell splits, ultimately deducing equivariant formality results analogous to (the additive part) of Theorem \ref{corollary: KH of simple varieties}.

While the main examples of our interest also have affine pavings, their singularities render the above techniques useless. To start with, even the additive computation outlined above does not seem obvious for equivariant $KH$ in the singular case, not to even mention the ring structures. 

Even more severe discrepancy between the methods occurs for our results about traces \S \ref{section: the p-adic cyclotomic trace}, \S \ref{section: rational goodwillie--jones trace in degree zero}: the statements in question do not hold for the affine spaces $\A^n$. Moreover, the Hoschschild-type invariants appearing as the target do not enjoy the relevant localization sequences. For these two reasons, there seems to be no hope for deducing our results by gluing from the affine strata. 

Finally, the class $\eC^G_k$ also covers some varieties without affine paving, such as Example \ref{example: projective nodal curve}.
\end{remark}

\subsection{Plan of the paper}
This paper is structured as follows, separating the arguments into three logically independent parts.

\noindent In \S \ref{section: localizing and truncating invariants of simple varieties} we define the classes of simple varieties $\eB^G_k \subseteq \eC^G_k$ by a constructive recipe, recall the necessary background on localizing invariants, and prove the main technical results -- Theorems \ref{theorem: vanishing of truncating invariants}, \ref{theorem: a degree zero version}, \ref{theorem: computation of truncating invariants} -- which allow to control equivariant truncating invariants on simple varieties.
\newline In \S \ref{section: applications to concrete localizing invariants} we discuss, one by one, the promised applications to various localizing invariants and equivalences between them -- see in particular Theorems \ref{corollary: K and TC}, \ref{theorem: rational goodwillie--jones trace in degree zero}, \ref{corollary: KH of simple varieties}. These three applications can be read independently.
\newline In \S \ref{section: examples of simple varieties}, we supply explicit examples of simple varieties, starting from elementary ones and ending with finite and affine Schubert varieties for $GL_n$ in Lemmata \ref{lemma: finite schubert varieties are simple}, \ref{example: affine schubert varieties are simple}. These final examples are truly interesting in geometric representation theory; all the results from the second section apply to them.

The paper contains two appendices. Appendix \ref{appendix: cdh sheaffified negative cyclic homology} reviews the supplementary localizing invariant $\LHCn(-/ k)$, relevant in \S \ref{section: rational goodwillie--jones trace in degree zero}. Appendix \ref{appendix: actions of linearly reductive G} outlines the variant of our results for actions of a linearly reductive group $G$.

\subsection{Acknowledgements}
First and foremost, I would like to thank Matthew Morrow for discussions, explanations and ideas without which this work would not have been carried out. I would further like to thank Brian Conrad for providing an amazing reference on projective cones in proper generality, to Vova Sosnilo for carefully discussing -- among other things -- the derived nilinvariance for quotients by any linearly reductive groups, and to Adeel Khan, Timo Richarz, Matthias Wendt and Xinwen Zhu for helpful conversations about the results. I would moreover like to thank the referee for the very useful comments.

This work took place during my visit at Laboratoire de Mathématiques d’Orsay, funded by Erasmus+ staff mobility training, and at the Institute of Science and Technology Austria (ISTA) during my PhD. I was supported by the DOC Fellowship of the Austrian Academy of Sciences.

\section{Localizing and truncating invariants of simple varieties}\label{section: localizing and truncating invariants of simple varieties}
We start by giving a constructive definition of the classes $\eB^G_k \subseteq \eC^G_k$ of simple varieties in \S \ref{section: a class of simple varieties}, which are the main object of study in this paper. We continue by recalling the notion of localizing and truncating invariants in \S \ref{section: localizing and truncating invariants} as well as some group-theoretic preliminaries in \S \ref{section: an interlude on groups and representations}. We then prove the technical result, Theorem \ref{theorem: vanishing of truncating invariants} in \S \ref{section: truncating invariants of simple varieties}, allowing us to control equivariant truncating invariants on simple varieties from $\eC^G_k$. We close the section by useful variants valid on the smaller class $\eB^G_k$, namely Theorems \ref{theorem: a degree zero version} and \ref{theorem: computation of truncating invariants} in \S \ref{section: degreewise versions}.

\subsection{A class of simple varieties}\label{section: a class of simple varieties}

\begin{notation}\label{notation: rings and groups}
Let $R$ be an arbitrary base ring, we write $S= \Spec R$ and denote $\Sch_R$ the category of quasi-compact separated schemes over $R$. 
Let $G$ be an affine group scheme over $R$. We then denote $\Sch^G_R$ the category of $G$-equivariant quasi-compact separated schemes over $k$ and equivariant maps between them. 
Any object $X$ of $\Sch^G_R$ gives rise to the global quotient stack $X/G$. Then $\Vect(X/G)$ and $\Perf(X/G)$ denote the categories of $G$-equivariant vector bundles and $G$-equivariant perfect complexes on $X$. We denote $BG = S/G$ the classifying stack of $G$ over $S$.

Using the above notation, we will mostly work over a base field $k$. We then write $\pt = \Spec k$.
\end{notation}

\begin{notation}\label{notation: flag bundles}
Given $X \in \Sch_k$, $\eF \in \Qcoh(X)$ and a dimension vector $d_{\bullet} = (d_1, \dots, d_m)$ of lenght $m$ and total dimension $d := d_1+ \dots + d_m$, we write $\Flag_X(\eF, d_{\bullet})$ for the relative flag bundle of type $d_{\bullet}$ parametrizing flags of vector bundle quotients of $\eF$ with successive subquotients of ranks $d_1, \dots, d_m$. When $m=1$, we get $\Grass_X(\eF, d)$.
\end{notation}

Let $G$ be any affine group scheme over $k$. We give the constructive definition of simple $G$-equivariant varieties straightaway. We emphasize that this definition stays within the realm of classical algebraic geometry.
\begin{definition}[Simple varieties $\eB^G_k$]\label{definition: closure properties on schemes}
Let $\eB^G_k \subseteq \Sch^G_k$ be the smallest full subcategory satisfying the following closure properties:
\begin{enumerate}
    \item \label{closure property: the point} (the point) The point $\pt$ with the trivial $G$ action lies in $\eB^G_k$. Moreover, $\eB^G_k$ is closed under finite disjoint unions. 
    \item \label{closure property: projective bundles} (projective bundles) Let $X \in \Sch^G_k$ and $\eE \in \Vect(X/G)$ be an equivariant vector bundle.
    \begin{enumerate}
        \item \label{closure property: projective bundles part a} If $\eE$ has rank $\geq 1$ everywhere\footnote{The rank conditions here and below are equivalent to the associated projectivization having nonempty fibers.}, then 
        $$X \in \eB^G_k \iff \P_X(\eE) \in \eB^G_k.$$
        \item \label{closure property: projective bundles part b} More generally, let $d_{\bullet}$ be any dimension vector of total dimension $d$ and assume that $\eE$ has rank $\geq d$ everywhere. Then
        $$X \in \eB^G_k \iff \Flag_X(\eE, d_{\bullet}) \in \eB^G_k$$
    \end{enumerate}
    \item \label{closure property: stratified projective bundles} (stratified projective bundles with nonempty fibers) Let $X \in \Sch^G_k$ and $\eF \in \Qcoh(X/G)$ be an equivariant quasi-coherent sheaf on $X$ that can be written as a cokernel of a map of equivariant vector bundles. 
    \begin{enumerate}
        \item \label{closure property: stratified projective bundles part a} Assume that $\eF$ has rank $\geq 1$ everywhere. Then
            \begin{equation*}
             \P_X(\eF) \in \eB^G_k \implies X \in \eB^G_k.    
             \end{equation*}
        \item \label{closure property: stratified projective bundles part b} More generally, let $d_{\bullet}$ be any dimension vector of total dimension $d$ and assume that $\eF$ has rank $\geq d$ everywhere. Then  
    \begin{equation*}
        \Flag_X(\eF, d_{\bullet}) \in \eB^G_k \implies X \in \eB^G_k.    
    \end{equation*}
    \end{enumerate}
    \item \label{closure property: 3-out-of-4 for split abstract blowups} (3-out-of-4 for split abstract blowups) Take a noetherian abstract blowup square $(X, Y, Z, E)$ in $\Sch^G_k$, meaning an equivariant pullback square of noetherian schemes over $k$
    \begin{equation*}\label{equation: equivariant abstract blowup square}
    \begin{tikzcd}
        Y \arrow[d] \arrow[r, hookleftarrow] & E \arrow[d] \\
        X \arrow[r, hookleftarrow] & Z,
    \end{tikzcd}
    \end{equation*}  
    where $X \hookleftarrow Z$ is a closed immersion while $Y \to X$ is proper and isomorphism over $X \setminus Z$. 
    Assume that the square is {\it split} in the sense that either $E \to Y$ admits a retraction or $Y \to X$ admits a section.
    If three of its terms lie in $\eB^G_k$, then the fourth term lies in $\eB^G_k$ as well.
    \end{enumerate}    
\end{definition} 

\begin{definition}[Simple varieties $\eC^G_k$]\label{definition: simple varieties C} 
We define a bigger class $\eC^G_k \supseteq \eB^G_k$ by the closure properties \eqref{closure property: the point}, \eqref{closure property: projective bundles}, \eqref{closure property: stratified projective bundles} and:
    \begin{enumerate}
    \item[(4')]\label{closure property: 3-out-of-4 for abstract blowups} (3-out-of-4 for abstract blowups) Take a noetherian abstract blowup square $(X, Y, Z, E)$ in $\Sch^G_k$.
    If three of its terms lie in $\eC^G_k$, then the fourth term lies in $\eC^G_k$ as well.
\end{enumerate} 
\end{definition} 

\begin{notation}
In particular, for trivial $G$ we get the subcategories $\eB_k \subseteq \eC_k$ of $\Sch_k$. Already this case is interesting.
\end{notation}

\begin{remark}\label{remark: classes of simple varieties for different groups}
If $G' \to G$ is a group homomorphism and $X \in \eB^G_k$, we get $X \in \eB^{G'}_k$ by restricting the action. In particular, $X \in \eB_k$. Similarly for the bigger class $\eC^G_k$. 
\end{remark}

In fact, the above definitions may be shortened by the following purely geometric argument.
\begin{lemma}\label{lemma: ommiting flag schemes}
The closure property \eqref{closure property: projective bundles part a} implies \eqref{closure property: projective bundles part b}. The closure properties \eqref{closure property: projective bundles part a} and \eqref{closure property: stratified projective bundles part a} jointly imply \eqref{closure property: stratified projective bundles part b}. 
Consequently, \eqref{closure property: projective bundles part b} and \eqref{closure property: stratified projective bundles part b} are automatic and can be omitted from the definition.
\end{lemma}
\begin{proof}
To show that \eqref{closure property: projective bundles part a} implies \eqref{closure property: projective bundles part b}, consider the corresponding full flag variety $\Flag_{X}(\eE, d'_{\bullet})$ where $d'_{\bullet} = (1, \dots, 1)$ of total dimension $d$. Then we have natural maps
\begin{equation*}
    \begin{tikzcd}
        \Flag_{X}(\eE, d'_{\bullet}) \arrow[bend left=20, rr, "h"] \arrow[swap]{r}{g} & \Flag_{X}(\eE, d_{\bullet}) \arrow[swap]{r}{f} & X
    \end{tikzcd}
\end{equation*}
and both $g$ and $h$ factor as towers of equivariant projective bundles, reducing to that case.


To see that \eqref{closure property: projective bundles part a} and \eqref{closure property: stratified projective bundles part a} jointly imply \eqref{closure property: stratified projective bundles part b}, consider the diagram
\begin{equation*}
    \begin{tikzcd}
        \Flag_X(\eF, d'_{\bullet}) \arrow[bend left=20, rr, "h"] \arrow[swap]{r}{g} & \Flag_X(\eF, d_{\bullet}) \arrow[swap]{r}{f} & X
    \end{tikzcd}
\end{equation*}
where $\Flag_X(\eF, d'_{\bullet})$ is the stratified full flag variety bundle for $d' = (1, \dots, 1)$ of total dimension $d$.
Now $g$ is a tower of honest projective bundles, so we know $\Flag_X(\eF, d_{\bullet}) \in \eB^G_k$ if and only if $\Flag_X(\eF, d'_{\bullet}) \in \eB^G_k$ by iterative use of \eqref{closure property: projective bundles part a} alone.
Moreover, $h$ is a tower of stratified projective bundles, so we already know that $\Flag_X(\eF, d'_{\bullet}) \in \eB^G_k \implies X \in \eB^G_k$ by iterative use of \eqref{closure property: stratified projective bundles part a} alone. Hence we conclude.
\end{proof}

\begin{remark}\label{remark: independence on nilpotent structures}
Given a noetherian scheme $X \in \Sch_k$ with underlying reduced subscheme $X_{\red} \hookrightarrow X$, we have $X \in \eB_k \iff X_{\red} \in \eB_k$. Similarly for the class $\eC_k$. 
Indeed, this is a direct consequence of \eqref{closure property: 3-out-of-4 for split abstract blowups} -- the square $(X, Y, Z, E) = (X, \emptyset, X_{\red}, \emptyset)$ is an abstract blowup; it is split with respect to the map $\emptyset \to \emptyset$. 

(In the more general setup of derived stacks, the independence on nilpotent and derived structures may be conveniently built directly into the definition; see Definition \ref{definition: the absolute class Cp} for such a formulation.)
\end{remark}

\begin{remark}
Since Definitions \ref{definition: closure properties on schemes}, \ref{definition: simple varieties C} are sufficient for our applications, we in general stick to them to minimize the technical load. 

Nevertheless, they can be further optimized.
Depending on the concrete application, we can even add more base cases to \eqref{closure property: the point}. For example, see Definition \ref{definition: the absolute class Cp} in \S \ref{section: strictly henselian base rings} for an absolute class $\eC_p$ depending only on a prime $p$ relevant for the cyclotomic trace, or Remark \ref{remark: KH arbitrary base} for a discussion of $KH$ of simple varieties over an arbitrary base ring $R$.
\end{remark}

\subsection{Localizing and truncating invariants}\label{section: localizing and truncating invariants}
We now recall the invariants we wish to control.
Let $E$ be any $k$-linear {\it localizing invariant} valued in spectra in the sense of \cite[Definition 1.2 and Remark 1.18]{LT19}. By definition, $E(-)$ is a functor of $\infty$-categories
\begin{equation*}
   E(-) : \Cat^{\perf}_k \to \Sp
\end{equation*}
which sends exact sequences to fiber sequences. In particular, if $\eD \in \Cat^{\perf}_k$ has a semiorthogonal decomposition $\eD = \langle \eD_j \rangle_{ j=0}^m$, we get a natural decomposition $E(\eD) \simeq \bigoplus_{j=0}^m E(\eD_j)$.

A localizing invariant is called {\it finitary} if it further commutes with filtered colimits. Many important localizing invariants have this property (and sometimes it is included into the definition). Also see \cite{Tam17, LT19, HSS17, BGT13}.

A localizing invariant $E$ is called {\it truncating} if it satisfies \cite[Definition 3.1]{LT19}: for every connective $\mathbb{E}_1$-ring spectrum $A$, the canonical map $E(A) \to E(\pi_0(A))$ is an equivalence. This further forces nilinvariance and cdh descent in various degrees of generality \cite{LT19, KR21, ES21}.

By convention, we evaluate $E$ on any derived stack $X$ via its category of perfect complexes
\begin{equation*}
    E(X) := E(\Perf(X)).
\end{equation*}
In particular, given any $G$-equivariant scheme $X \in \Sch_k^G$, we denote
\begin{equation*}
    E^G(X) := E(X/G) := E(\Perf(X/G))
\end{equation*}
and freely switch between these notations.

For any derived stack $X$, the category $\Perf(X)$ carries a symmetric monoidal structure $\otimes$. All the localizing invariants we consider are lax monoidal: they in particular turn this into an $\mathbb{E}_{\infty}$-ring structure on the outputted spectrum $E(X)$, whose homotopy groups thus form a graded-commutative ring $E_{\bullet}(X)$. For brevity, we call such localizing invariants {\it multiplicative}.

\begin{recollection}
The following are $k$-linear localizing invariants:
\begin{itemize}
    \item non-connective algebraic $K$-theory $K(-)$, it is finitary,
    \item topological Hochschild homology $THH(-)$, it is finitary,
    \item Hochschild homology of $k$-linear categories $HH(-/ k)$, it is finitary,
    \item negative cyclic homology of $k$-linear categories $HC^{-}(-/ k)$,
    \item topological cyclic homology $TC(-)$.
\end{itemize}
Furthermore, the following are $k$-linear truncating invariants:
\begin{itemize}
    \item homotopy $K$-theory $KH(-)$ of $k$-linear categories, it is finitary,
    \item the fiber of the cyclotomic trace $K^{\inf}(-)$,
    \item periodic cyclic homology of $k$-linear categories $HP(-/ k)$ if $\chara k = 0$,
    \item topological $K$-theory of $\C$-linear categories $K_{\top}(-)$, it is finitary.
\end{itemize}
\end{recollection}
\begin{proof}
For $K(-)$, $THH(-)$, see \cite{BGT13}, \cite[Example 17]{Tam17}. For $HH(-/ k)$, see \cite{Hoy18}, \cite{HSS17}; it commutes with filtered colimits by definition. Further see \cite[p. 914]{LT19} for a discussion of $HH(-/k)$, $HC^{-}(-/k)$, $HP(-/ k)$ as localizing invariants and \cite[proof of Corollary 3.6]{LT19} for $TC(-)$ and the truncating invariant $K^{\inf}(-)$. For $K_{\top}(-)$ as a finitary localizing invariant see \cite[Theorem 1.1]{Bla16}; it is truncating by \cite{Kon21}.

Also $KH(-)$ is a truncating invariant by \cite[Proposition 3.14]{LT19}; it commutes with filtered colimits as $K(-)$ does by swapping colimits in its definition.
\end{proof}

\begin{notation}
We use the following notation:
\begin{itemize}
    \item $K(-; \F_p) := K(-)/p$ for the mod-$p$ algebraic $K$-theory,
    \item $K(-; \Z/p^j):= K(-)/p^j$ and $K(-; \Z_p) := \lim_{j} K(-; \Z/p^j)$ for its mod-$p^j$ and $p$-adic versions,
    \item $K(-; \Q):= K(-)_{\Q}$ for rationalized $K$-theory.
\end{itemize}
See \cite[\S 9.3]{TT90} for details.
Given any localizing invariant $E$ in place of $K$-theory, we use the analogous notations
$E(-; \F_p)$, $E(-; \Z/p^j)$, $E(-; \Z_p)$ and $E(-; \Q)$.
\end{notation}

The above discussion also works over an arbitrary base ring $R$ in place of $k$; we will use this in a few side remarks.

\subsection{An interlude on groups and representations}\label{section: an interlude on groups and representations}
We now discuss the groups that we eventually want to consider. The results of this paper work best for nice groups over fields. They further generalize to linearly reductive groups over fields (giving broader applicability in characteristic zero). In a different direction, restricting to diagonalizable groups only, our results work over more general base rings.

\subsubsection{Nice and diagonalizable groups.}\label{section: nice and diagonalizable groups}

Let $M$ be an abstract abelian group. Then its group algebra $R[M]$ is naturally a commutative and cocommutative Hopf algebra, and we write $D(M) = \Spec(R[M])$ for the corresponding affine group scheme. An affine group scheme $G$ over $R$ is called \textit{diagonalizable} if it is of the form $D(M)$ for some abelian group $M$. 

The diagonalizable group $D(M)$ is of finite type if and only if $M$ is finitely generated. By the classification of finitely generated abelian groups, diagonalizable groups of finite type are given by finite products of the multiplicative group $\Gm$ and the groups of finite order roots of unity $\mu_{p^{a}}$ of a prime power $p^a$.

An affine group scheme $G$ over $R$ is of {\it multiplicative type} if it is diagonalizable fpqc locally on $R$. 

An affine group scheme $G$ of finite type over $R$ is called {\it nice}, if it is an extension of a closed and open normal subgroup $G^0$ of multiplicative type by an étale group scheme of order invertible in $R$. It is called {\it linearly reductive} if the functor $(-)^G$ of invariants is exact. Nice groups are always linearly reductive.

\subsubsection{Linearly reductive groups over fields.}
Let $G$ be a linearly reductive affine group scheme of finite type over a field $k$. 
In concrete terms \cite[Remark 2.3]{AHR19}, such $G$ is given as follows:
\begin{itemize}
    \item[(i)] if $\chara k = 0$, this is equivalent to saying that $G$ is an extension of a finite group by a connected reductive group,
    \item[(ii)] if $\chara k = p$ is positive, this is equivalent to saying that $G$ is nice, i.e. an extension of an étale group whose order is not divisible by $p$ by a group of multiplicative type.
\end{itemize}

In characteristic $p$, the notions of nice and linearly reductive groups coincide. In characteristic zero, reductive groups form a bigger class (e.g. $GL_n$ is linearly reductive but not nice).

\subsubsection{Representations and decomposability.}\label{section: representations and decomposability}
In order to compute equivariant localizing invariants of the point, we need to control the category $\Perf(BG)$. We denote by $\Rep_k(G)$ the category of finite dimensional algebraic representations of $G$; it is tautologically equivalent to the category $\Vect(BG)$ of vector bundles on its classifying stack. We write $R(G) = K_0(BG)$ for the representation ring of $G$.

The above categories become very easy in the following naturally occurring situation.
\begin{discussion}\label{discussion: decomposition of BG}\label{discussion: decomposition of BT}
We say that an affine group scheme $G$ over a ring $R$ has {\it decomposable representation theory} if there is an index set $I$ and an equivalence in $\Cat^{\perf}_R$   
\begin{equation}\label{equation: decomposition of BG}
\Perf(R/G) \simeq \bigoplus_{I} \Perf(R).   
\end{equation}
The following groups give the most notable examples.
\begin{itemize}
    \item[(i)] Let $G$ be a linearly reductive group over a field $k$. Then the category $\Rep_k(G)$ is semisimple; it is an infinite direct sum of copies of $\Vect(\pt)$ labeled by the irreducibles in $\Rep_k(G)$. Consequently, $G$ has decomposable representation theory with
    $\Perf(BG) \simeq \bigoplus_{\Irr(G)} \Perf(k)$.  
    \item[(ii)] Let $G = D(M)$ is a diagonalizable group over $ S = \Spec R$. Representations of $G$ correspond to $M$-graded quasi-coherent sheaves on $S$. The latter abelian category decomposes as a direct product of copies of quasi-coherent sheaves on $R$ over the index set $M$. Consequently, $G$ has decomposable representation theory with $\Perf(S/G) \simeq \bigoplus_{M} \Perf(S)$.
\end{itemize}
In these examples, the natural monoidal structure on the left-hand side corresponds to the convolution on the right-hand side.     
\end{discussion}

Decomposition \eqref{equation: decomposition of BG} has the following simple consequence: the value of finitary localizing invariants on the equivariant base is determined by its value on $R$.
\begin{lemma}\label{lemma: value on equivariant point}
Let $G$ be a group over $R$ with decomposable representation theory and $E$ a finitary localizing invariant. Then
\begin{equation*}
    E^G(R) \simeq \bigoplus_{I} E(R).
\end{equation*}    
\end{lemma}
\begin{proof}
Since $E(-)$ commutes with finite direct sums and filtered colimits, it is compatible with infinite direct sums, so we conclude from \eqref{equation: decomposition of BG}.
\end{proof}

\begin{remark}\label{remark: finitary not so important}
In applications, one can often check the outcome of Lemma \ref{lemma: value on equivariant point} even in cases when $E$ is not finitary per se. More precisely, it is often sufficient to know that $E(\pt)$ is sufficiently bounded. See Lemmata \ref{lemma: cyclotomic trace on equivariant point} and \ref{lemma: vanishing on BT}. 
\end{remark}

\subsection{Truncating invariants of simple varieties}\label{section: truncating invariants of simple varieties}
Let $k$ be a field and $G$ a nice group over it. 
The following theorem then gives control over truncating invariants on $\eC^G_k$. It is one of the main technical observations of this paper. Also see \S \ref{section: degreewise versions} for further degreewise variations, and Appendix \ref{appendix: actions of linearly reductive G} for the case of linearly reductive $G$.

\begin{theorem}\label{theorem: vanishing of truncating invariants}
Let $E$ be a $k$-linear truncating invariant with $E^G(\pt) \simeq 0$. 
Then for all $X \in \eC^G_{k}$ we have
\begin{equation*}
    E^G(X) \simeq 0.
\end{equation*}
\end{theorem}
\begin{proof}
We check that $E^G(X) \simeq 0$ is stable under the closure properties from Definition \ref{definition: simple varieties C}.

\eqref{closure property: the point} By assumption, $E^G(\pt) \simeq 0$; the compatibility with finite disjoint unions is clear as $E$ commutes with finite direct sums. 

\eqref{closure property: projective bundles} The case \eqref{closure property: projective bundles part a} of $\P_X(\eE)$ follows from the projective bundle formula for localizing invariants. More precisely, since $X$ is quasi-compact by design, it has only finitely many connected components, so by \eqref{closure property: the point} we may assume $X$ is connected; we then denote $n := \rank \eE \geq 1$. Now, $\Perf(\P_{X/G}(\eE))$ has a semi-orthogonal decomposition into $n$ copies of $\Perf(X/G)$ by the classical results of \cite{Tho88}; see also \cite[Theorem 3.3 and Corollary 3.6]{Kha18}. Since localizing invariants send semi-orthogonal decompositions to direct sums, we have $E^G(\P_X(\eE)) \simeq \bigoplus_{i=0}^{n-1} E^G(X)$, giving the desired equivalence.
The more general case \eqref{closure property: projective bundles part b} of $\Flag_{X}(\eE, d_{\bullet})$ is then automatic by Lemma \ref{lemma: ommiting flag schemes}.

\eqref{closure property: stratified projective bundles} By the assumption on $\eF$, we can find a two term complex of equivariant vector bundles $\eE_{\bullet} = [\eE_1 \to \eE_0]$ with $\coker(\eE_1 \to \eE_0) \cong \eF$. To prove \eqref{closure property: stratified projective bundles part a}, take its derived projectivization $Y = \P_{X}(\eE_{\bullet})$ in the sense of \cite{Jia22a, Jia22b}. This is a derived scheme with classical truncation $Y_{\cl} = \P_{X}(\eF)$. Since $\eE_{\bullet}$ has Tor-amplitude $[1, 0]$, we obtain the semi-orthogonal decomposition of \cite[Theorem 3.2 and Remark 1.1]{Jia23}, \cite[Theorem 7.5]{Jia22a}. In particular, as the rank of $\eF$ is $\geq 1$ everywhere, the pullback realizes $\Perf(X/G)$ as a semiorthogonal summand in $\Perf(Y/G)$ -- note that this is independent of the characteristic of $k$ (in fact works over any base) by \cite[Remark 1.1]{Jia23}, \cite[Theorem 7.5]{Jia22a} since we are treating only derived projectivizations at the moment. Since localizing invariants send semi-orthogonal decompositions to direct sums, $E^G(Y) \simeq 0 \implies E^G(X) \simeq 0$. Moreover, as $E(-)$ is truncating, we have $E^G(Y_{\cl}) \simeq E^G(Y)$ by \cite[Corollary 5.2.3]{ES21}. Altogether, this proves the first statement.
The more general case \eqref{closure property: stratified projective bundles part b} of a stratified partial flag variety $\Flag_X(\eF, d_{\bullet})$ is again automatic by Lemma \ref{lemma: ommiting flag schemes}.

\eqref{closure property: 3-out-of-4 for abstract blowups} Any truncating invariant sends $G$-equivariant abstract blowup squares to homotopy fiber squares by \cite[Corollary 5.2.6]{ES21} -- the statement follows.
\end{proof}

\begin{remark}
The discussion of \eqref{closure property: the point} and \eqref{closure property: projective bundles} clearly works for any localizing invariant. However, for both \eqref{closure property: stratified projective bundles} and \eqref{closure property: 3-out-of-4 for abstract blowups}, it is crucial that $E$ is truncating. The closure properties \eqref{closure property: stratified projective bundles} and \eqref{closure property: 3-out-of-4 for abstract blowups} are important, as they allow us to go down and reach singular varieties.     
\end{remark}

\begin{corollary}\label{corollary: truncating invariants which commute with filtered colimits}
Let $E \to F$ be a map of $k$-linear finitary truncating invariants. If it induces an equivalence $E(\pt) \simeq F(\pt)$, then it induces an equivalence $E^G(X) \simeq F^G(X)$ for any $X \in \eC^G_k$.   
\end{corollary}
\begin{proof}
The fiber $\fib(E \to F)$ is a $k$-linear finitary truncating invariant which vanishes on $\pt$. We conclude by Lemma \ref{lemma: value on equivariant point} and Theorem \ref{theorem: vanishing of truncating invariants}.
\end{proof}

The assumption that $E$, $F$ are finitary is used solely to commute $E$ through the decomposition \eqref{equation: decomposition of BG} via Lemma \ref{lemma: value on equivariant point}. As we already mentioned in Remark \ref{remark: finitary not so important}, this can often be done by hand in more general situations.

\subsection{Degreewise versions}\label{section: degreewise versions}
If we restrict attention only to the smaller class $\eB^G_k$, Theorem \ref{theorem: vanishing of truncating invariants} works degreewise. This is a very useful variant -- there are interesting maps of localizing invariants which induce isomorphisms only in certain degrees. 

As before, assume $G$ is a nice group over a field $k$. Also see Appendix \ref{appendix: actions of linearly reductive G} for the case of linearly reductive $G$. 
\begin{theorem}\label{theorem: a degree zero version}
Fix $i \in \Z$ and let $E$ be a $k$-linear truncating invariant with $E^G_i(\pt) \simeq 0$. Then for all $X \in \eB^G_{k}$ we have
\begin{equation*}
    E^G_i(X) \simeq 0.
\end{equation*}
More generally, fix $i \in \Z$ and let $E$ and $F$ be $k$-linear truncating invariants together with a map $E \to F$ which induces an isomorphism $E^G_i(\pt) \simeq F^G_i(\pt)$. Then for all $X \in \eB^G_{k}$ we have
\begin{equation*}
    E^G_i(X) \simeq F^G_i(X).
\end{equation*}
\end{theorem}
\begin{proof}
We focus on the latter statement, the first being a special case for $F$ trivial.
It is enough to check that $E^G_i(X) \simeq F^G_i(X)$ is stable under the closure properties from Definition \ref{definition: closure properties on schemes}. This follows by the arguments from the proof of Theorem \ref{theorem: vanishing of truncating invariants}.
Indeed, the desired isomorphism is stable under taking direct sums and direct summands in $\Cat^{\perf}_k$, so it is stable under \eqref{closure property: the point} and \eqref{closure property: projective bundles}. Since $E_i(-)$ and $F_i(-)$ disregard derived structures on stacks in the necessary generality \cite[Corollary 5.2.3]{ES21}, the compatibility with direct summands shows the stability under \eqref{closure property: stratified projective bundles}. Since $E_i(-)$ and $F_i(-)$ turn split abstract blowup squares of stacks into split homotopy fiber squares in the necessary generality \cite[Corollary 5.2.6]{ES21}, we deduce the stability on \eqref{closure property: 3-out-of-4 for split abstract blowups}.
\end{proof}

In a similar vein, the value of any multiplicative truncating invariant $E$ on some $X \in \eB^G_k$ is determined by the degree zero part $E^G_0(X)$, which is well-controlled.

\begin{theorem}\label{theorem: computation of truncating invariants}
Let $E$ be a $k$-linear finitary truncating invariant, which is multiplicative. Then for all $X \in \eB^G_k$ the natural map induces a graded ring isomorphism
\begin{equation}\label{equation: computation of truncating invariants}
E^G_0(X) \underset{E_0(\pt)}{\otimes} E_{\bullet}(\pt) \xrightarrow{\cong}  E^G_{\bullet}(X).
\end{equation}
Furthermore, the commutative ring $E^G_0(X)$ is a finite projective module over $E^G_0(\pt)$.
\end{theorem}
\begin{proof}
Since $E$ is multiplicative, the universal property of tensor products induces is a natural ring homomorphism 
\begin{equation}\label{equation: natural tensor map on homotopy groups}
E_0(-)\underset{E_0(\pt)}{\otimes} E_{\bullet}(\pt) \to E_{\bullet}(-)    
\end{equation}
and we only need to check that it is an isomorphism on the underlying abelian groups. We check that this statement is stable under the closure properties from Definition \ref{definition: closure properties on schemes}.

For non-equivariant $\pt$, the map \eqref{equation: natural tensor map on homotopy groups} is tautologically an isomorphism. Since $E$ is finitary, both sides of \eqref{equation: natural tensor map on homotopy groups} commute with the infinite direct sum \eqref{equation: decomposition of BG}, giving
\begin{equation}\label{equation: equivariant versus non-equivariant}
E^G_0(\pt) \underset{E_0(\pt)}{\otimes} E_{\bullet}(\pt) \xrightarrow{\cong}  E^G_{\bullet}(\pt). 
\end{equation}
For the same reason, the statement is stable under disjoint unions. Altogether, we have checked the closure property \eqref{closure property: the point}.

Given $\eD \in \Cat^{\perf}_k$ with a semi-orthogonal decomposition $\eD = \langle \eD_j \rangle_{ j=0}^m$ and $i \in \Z$, we have
\begin{equation*}
E_i(\eD) = \bigoplus_{j =0}^m E_i(\eD_j), \qquad \text{hence} \qquad  E_0(\eD)\underset{E_0(\pt)}{\otimes} E_{\bullet}(\pt) = \bigoplus_{j = 0}^m E_0(\eD_j)\underset{E_0(\pt)}{\otimes} E_{\bullet}(\pt),  
\end{equation*}
so both sides of \eqref{equation: natural tensor map on homotopy groups} send semi-orthogonal decompositions to direct sums.

In particular, stability under \eqref{closure property: projective bundles} follows from the semi-orthogonal decompositions for projective bundles. Since $E$ is truncating, the closure property \eqref{closure property: stratified projective bundles} follows from the the semi-orthogonal decompositions of derived projectivizations \cite{Jia23} together with the independence of $E_i(-)$ on derived structures \cite[Corollary 5.2.3]{ES21}.

Finally, since $E$ is truncating, any noetherian split abstract blowup $(X, Y, Z, E)$ induces a long exact sequence on homotopy groups of $E$ by \cite[Corollary 5.2.3]{ES21}, which falls apart into split short exact sequences
\begin{equation*}
    0 \to E^G_i(X) \to E^G_i(Y) \oplus E^G_i(Z) \to E^G_i(E) \to 0, \qquad i\in \Z. 
\end{equation*}
In particular, we get the split short exact sequence
\begin{equation*}
    0 \to E^G_0(X)\underset{E_0(\pt)}{\otimes} E_{\bullet}(\pt) \to (E^G_0(Y)\underset{E_0(\pt)}{\otimes} E_{\bullet}(\pt)) \oplus (E^G_0(Z)\underset{E_0(\pt)}{\otimes} E_{\bullet}(\pt) ) \to E^G_0(E)\underset{E_0(\pt)}{\otimes} E_{\bullet}(\pt) \to 0. 
\end{equation*}
Altogether, both sides of \eqref{equation: natural tensor map on homotopy groups} behave compatibly under the closure property \eqref{closure property: 3-out-of-4 for split abstract blowups}.

Finally, in each step \eqref{closure property: the point}-\eqref{closure property: 3-out-of-4 for split abstract blowups} we are only passing to finite direct sums or retracts. It is thus clear that the property of $E^G_0(X)$ being finite projective over $E^G_0(\pt)$ is preserved.
\end{proof}

\begin{remark}
The statements of Theorems \ref{theorem: a degree zero version}, \ref{theorem: computation of truncating invariants} are not true for the bigger class $\eC_k^G$ by Example \ref{example: projective nodal curve}.   
\end{remark}

\begin{remark}
The isomorphism from Theorem \ref{theorem: computation of truncating invariants} refines to
\begin{equation*}\label{equation: computation of truncating invariants II}
E^G_0(X) \underset{E_0(\pt)}{\otimes} E_{\bullet}(\pt) \xrightarrow{\cong} E^G_0(X) \underset{E^G_0(\pt)}{\otimes} E^G_{\bullet}(\pt) \xrightarrow{\cong} E^G_{\bullet}(X).
\end{equation*}
Indeed, we have natural maps as indicated; the first one is also an isomorphism by \eqref{equation: equivariant versus non-equivariant}.
\end{remark}

For the next remark, recall that a map of $\mathbb{E}_{\infty}$-ring spectra $A \to B$ is called {\it (finite) faithfully flat} if the map of commutative rings $\pi_0(A) \to \pi_0(B)$ is (finite) faithfully flat and $\pi_{\bullet}(B) \cong \pi_{0}(B) \otimes_{\pi_{0}(A)} \pi_{\bullet}(A)$. In this language, we can rephrase the above results as follows.
\begin{remark}
Concisely, $E^G(X)$ is faithfully flat over $E(\pt)$. It is finite faithfully flat over $E^G(\pt)$.
\end{remark}

\section{Applications to concrete localizing invariants}\label{section: applications to concrete localizing invariants}
We now present several instances of the above theorems applied to concrete localizing invariants. The most interesting applications concern global situations in positive and mixed characteristic where the $p$-adic cyclotomic trace becomes an equivalence \S \ref{section: the p-adic cyclotomic trace}; situations over the rationals where the Goodwillie--Jones trace induces an isomorphism in degree zero \S \ref{section: rational goodwillie--jones trace in degree zero}; and situations where homotopy-invariant $K$-theory becomes equivariantly formal \S \ref{section: homotopy invariant K-theory and equivariant formality}. The lastly mentioned formality gives, for instance, a way to compute $KH$ of complex simple varieties from their topological $K$-theory $K_{\top}$.

\subsection{The \texorpdfstring{$p$}{p}-adic cyclotomic trace}\label{section: the p-adic cyclotomic trace}
As a first application, we show in Theorem \ref{corollary: K and TC} that the mod-$p$ cyclotomic trace is an equivalence on simple varieties with nice group actions; this already covers the geometric examples of our interest.  

We then explain a mild generalization of the result including certain henselian situations in Theorem \ref{theorem: p-adic cyclotomic trace of Cp} -- a reader interested in varieties over a field can skip this variant.

\begin{setup}
Throughout this section, assume $k$ is an algebraically closed field of characteristic $p$. Let $G$ be a nice group scheme over $k$.
\end{setup}

\subsubsection{The \texorpdfstring{$p$}{p}-adic cyclotomic trace on simple varieties.}\label{section: the p-adic cyclotomic trace on simple varieties}
We give the following application of Theorem \ref{theorem: vanishing of truncating invariants} to simple varieties from $\eC^G_k$.

\begin{theorem}[Mod-$p$ cyclotomic trace]\label{corollary: K and TC}
For any $X \in \eC^G_{k}$, the mod-$p$ cyclotomic trace induces an equivalence
\begin{equation*}
    K^G(X; \F_p) \xrightarrow{\simeq} TC^G(X; \F_p).
\end{equation*}
\end{theorem}
\begin{proof}
Note that the fiber $K^{\inf}(-; \F_p)$ is a truncating invariant. We check that $K^{\inf}(BG; \F_p) \simeq 0$ in Lemma \ref{lemma: cyclotomic trace on equivariant point} below, hence $K^{\inf}(X/G; \F_p)\simeq 0$ for any $X \in \eC^G_k$ by Theorem \ref{theorem: vanishing of truncating invariants} as desired.
\end{proof}

\begin{remark}[$p$-adic cyclotomic trace]
The same conclusion formally follows for mod-$p^j$ and $p$-adic versions: for any $X \in \eC^G_k$, we have 
\begin{equation*}
\forall j \in \N, \ K^G(X; \Z/p^j) \xrightarrow{\simeq} TC^G(X; \Z/p^j) 
\qquad \text{and} \qquad 
K^G(X; \Z_p) \xrightarrow{\simeq} TC^G(X; \Z_p).    
\end{equation*}    
\end{remark}

\begin{remark}[Non-equivariant case is interesting]\label{remark: non-equivariant case is interesting}
The above statement is interesting already non-equivariantly: for any $X \in \eC_k$, it gives an equivalence $K(X; \F_p) \xrightarrow{\simeq} TC(X; \F_p)$.
In general, it is known that the cyclotomic trace $K(X; \F_p) \to TC(X; \F_p)$ realizes the target as the étale sheafification of the source by \cite[Theorem 6.3]{CMM21}; also see \cite{GH99} for results in the smooth case. In other words, our Theorem \ref{corollary: K and TC} asserts that the étale sheafification is not necessary after evaluation on $X \in \eC_k$. To us, this is far from obvious: we do not know how to deduce our result by studying the étale site of $X$.
\end{remark}

To complete the proof of the theorem, we still need to supply the case of the equivariant point.
\begin{lemma}\label{lemma: cyclotomic trace on equivariant point}
We have an equivalence 
$$K^G(\pt; \F_p) \xrightarrow{\simeq} TC^G(\pt; \F_p).$$
In other words, $K^{\inf}(BG; \F_p) \simeq 0$.  
\end{lemma}
\begin{proof}
Since $k$ is an algebraically closed field of characteristic $p$, the cyclotomic trace
\begin{equation}\label{equation: K and TC of a point}
K(\pt; \F_p) \xrightarrow{\simeq} TC(\pt; \F_p)   
\end{equation}
is an equivalence. Indeed, $K(\pt)$ lives in degrees $\geq 0$ and all the higher $K$-groups are $\Z[\tfrac{1}{p}]$-modules \cite[Corollary 5.5. and Example (1) below it]{Kra80b}. Since $K_0(\pt) \cong \Z$, we see that $K(\pt; \F_p)$ is concentrated in degree $0$ with value $\F_p$. At the same time, since $k$ is algebraically closed, $TC(\pt; \F_p)$ also lives in degree $0$ with value $\F_p$ by \cite[Example 7.4 and Remark 7.6]{KN18}. The trace map clearly induces an isomorphism between them.

To treat the equivariant case, consider the decomposition \eqref{equation: decomposition of BG}. Since $THH(-)$ is finitary, we get an equivalence
\begin{equation}\label{equation: THH}
THH(BG) \simeq \bigoplus_{\Irr(G)} THH(\pt).    
\end{equation}

Note that $THH(k)$ is connective (this being the case for $THH$ of any connective ring spectrum in place of $k$), so both sides of \eqref{equation: THH} are connective. By general theory, $THH(-)$ carries a natural structure of a cyclotomic spectrum; the forgetful map from cyclotomic spectra to spectra preserves colimits \cite[\S 2.1, p. 416]{CMM21}. Altogether, \eqref{equation: THH} holds in the category $\CycSp$ of cyclotomic spectra, and hence in the category $\CycSp_{\geq 0}$ of connective cyclotomic spectra. Now, by \cite[Theorem 2.7]{CMM21}, the functor $TC(-; \F_p): \CycSp_{\geq 0} \to \Sp$ commutes with colimits. Altogether,
\begin{equation*}\label{equation: TC}
TC(BG; \F_p) \simeq \bigoplus_{\Irr(G)} TC(\pt; \F_p).    
\end{equation*}
Since $K$-theory is finitary, the decomposition \eqref{equation: decomposition of BG} also gives
\begin{equation*}\label{equation: K}
K(BG; \F_p) \simeq \bigoplus_{\Irr(G)} K(\pt; \F_p).    
\end{equation*}
All the above identifications are compatible with the trace map, so we reduce to the case of $\pt$ treated in \eqref{equation: K and TC of a point} above.
\end{proof}

\subsubsection{Strictly henselian base rings.}\label{section: strictly henselian base rings}
In fact, Theorem \ref{corollary: K and TC} can be easily generalized to include certain strictly henselian situations in both equal and mixed characteristic as base cases. To this end, we introduce the following absolute class of simple stacks $\eC_p$ depending only on a prime number $p$, which is given by allowing more general base examples in Definition \ref{definition: closure properties on schemes}. We use this as an opportunity to rephrase the rest of the construction in a stacky way.

We denote $\dStk^{\ans}$ the category of derived algebraic stacks with nice stabilizers, see \cite[\S A.1]{BKRS22}, \cite[\S 2.4 and \S 2.5]{KR21}. We also recall that the notions of abstract blowup squares \cite[Definition 2.1.3]{BKRS22}, Nisnevich squares \cite[Definition 2.1.1]{BKRS22}, derived projectivizations \cite{Jia22a, Jia22b, Jia23} work in this generality. 
We then define the class $\eC_p \subseteq \dStk^{\ans}$ as follows.

\begin{definition}\label{definition: the absolute class Cp}
Let $p$ be a fixed prime. Let $\eC_p \subseteq \dStk^{\ans}$ be the class of derived algebraic stacks given by the following closure properties:
\begin{enumerate}[start=0]
    \item \label{C_p 0} (independence on nilpotence) Let $\eX \in \dStk^{\ans}$. Then $\eX \in \eC_p \iff \eX_{\cl} \in \eC_p \iff \eX_{\red} \in \eC_p$.
    \item \label{C_p 1} \label{closure property: strictly henselian base rings} (base cases) Let $S = \Spec R$ be the spectrum of a strictly henselian, weakly regular, stably coherent ring $R$ of residue characteristic $p$. Let $G$ be a group scheme over $S$ with decomposable representation theory. Then $S / G \in \eC_p$. Moreover, $\eC_p$ is closed on finite disjoint unions.
    \item \label{C_p 2}  (projective bundles) Let $\eX \in \dStk^{\ans}$ and $\eE \in \Vect(\eX)$ of rank $\geq 1$ everywhere, then $\P_{\eX}(\eE) \in \eC_p \iff \eX \in \eC_p$. More generally, let $d_{\bullet}$ be a dimension vector of total dimension $d$ and assume that $\eE$ has rank $\geq d$ everywhere, then $\eX \in \eC_p \iff \Flag_{\eX}(\eE, d_{\bullet}) \in \eC_p$.
    \item \label{C_p 3} (stratified projective bundles) Let $\eX \in \dStk^{\ans}$ and $\eF \in \Qcoh(\eX)$ be of the form $\eF = H_0(\eE)$ for some $\eE \in \Perf^{\geq 0}(\eX)$. Assume $\eF$ has rank $\geq 1$ everywhere, then
    $\P_{\eX}(\eF) \in \eC_p \implies {\eX} \in \eC_p$. More generally, let $d_{\bullet}$ be a dimension vector of total dimension $d$ and assume that $\eF$ has rank $\geq d$ everywhere, then $\Flag_{\eX}(\eF, d_{\bullet}) \in \eC_p \implies {\eX} \in \eC_p$.
    \item \label{C_p 4}  (cdh descent) Given a noetherian abstract blowup square $(\mathscr{X}, \mathscr{Y}, \mathscr{Z}, \mathscr{E})$ in $\dStk^{\ans}$ such that three of its terms lie in $\eC_p$, the fourth one does as well. Similarly for arbitrary Nisnevich squares.
\end{enumerate}
\end{definition}

\begin{example}\label{example: stacks in Cp}
Let us mention some examples of stacks in $\eC_p$.
\begin{itemize}
    \item[(i)] If $G$ is a nice group over any algebraically closed field $k$ of characteristic $p$ and $X \in \eC^G_k$, then $\eX := X/G \in \eC_p$.  
    \item[(ii)] We allow more general affine base rings $R$ to lie in $\eC_p$ in the closure property \eqref{closure property: strictly henselian base rings}. 
    In particular, $R$ can be any regular noetherian strictly henselian ring such as the power series ring $k\llb t \rrb$ or the ring of integers $\breve{\Z}_p$ of the completion of the maximal unramified extension of $\Q_p$. More generally, $R$ can be any strictly henselian valuation ring by \cite[Corollary 2.3]{AMM22}.
    \item[(iii)]  The remaining closure properties then allow to combine the above: for example, the blowup of such regular noetherian strictly henselian $R$ in its special point again lies in $\eC_p$.
\end{itemize}
\end{example}

We then have the following generalization of Theorem \ref{corollary: K and TC} and Lemma \ref{lemma: cyclotomic trace on equivariant point} to the class $\eC_p$. 

\begin{theorem}\label{theorem: p-adic cyclotomic trace of Cp}
For any $\eX \in \eC_p$, the cyclotomic trace induces an equivalence
$$K(\eX; \F_p) \xrightarrow{\simeq} TC(\eX; \F_p).$$
Similarly for the mod-$p^j$ and $p$-adic variants.
\end{theorem}
\begin{proof}
The independence on nilpotent structures \eqref{C_p 0} is \cite{ES21}. The base cases \eqref{C_p 1} follow from Lemma \ref{lemma: strictly henselian base rings} below. Then \eqref{C_p 2} and \eqref{C_p 3} are same as in Theorem \ref{theorem: computation of truncating invariants} using that the semiorthogonal decompositions for projective bundles and derived projectivizations work over $\Z$, see \cite[Theorem 7.5]{Jia22a} and \cite[Remark 1.1]{Jia23}. Finally, \eqref{C_p 4} follows from \cite[Corollary 5.2.6]{ES21}.
\end{proof}

\begin{lemma}\label{lemma: strictly henselian base rings}
Let $S = \Spec R$ be the spectrum of a strictly henselian, weakly regular, stably coherent ring $R$ of residue characteristic $p$. Consider a group $G$ over $S$ with decomposable representation theory. Then the cyclotomic trace induces an equivalence
$$K(S/G; \F_p) \xrightarrow{\simeq} TC(S/G; \F_p).$$
In other words, $K^{\inf}(S/G; \F_p) \simeq 0$.   
\end{lemma}
\begin{proof}
For weakly regular, stably coherent rings $R$, the map $K^{\cn}(S) \to K(S)$ from connective $K$-theory to $K$-theory is an equivalence by \cite[Proposition 2.4]{AMM22}. Moreover, for strictly henselian rings $R$ of residue characteristic $p$, the map $K^{\cn}(S; \F_p) \to TC(S; \F_p)$ is an equivalence by \cite[Theorem 6.1]{CMM21}. Altogether, we have an equivalence
\begin{equation*}\label{equation: K and TC of R}
    K(S; \F_p) \xrightarrow{\simeq} TC(S; \F_p).
\end{equation*}
The rest of the proof of Lemma \ref{lemma: cyclotomic trace on equivariant point} then works word for word with $k$ replaced by $R$, using that $G$ has decomposable representation theory by assumption.
\end{proof}

\subsection{Rational Goodwillie--Jones trace}\label{section: rational goodwillie--jones trace in degree zero}

Consider the trace map $K(-; \Q) \to HC^{-}(- / \Q)$ from rational $K$-theory to $\Q$-linear negative cyclic homology. We prove that on $\eB_{\Q}$, the rational Goodwillie--Jones trace induces an isomorphism in degree zero in Theorem \ref{theorem: rational goodwillie--jones trace in degree zero}. We also deduce a counterintuitive direct sum decomposition for positive degrees in Proposition \ref{proposition: direct sum in positive degrees}.

\begin{setup}
Throughout this subsection we take $k = \Q$. Let $G$ be a nice group over $\Q$. Via Appendix \ref{appendix: actions of linearly reductive G}, the results generalize to any reductive group $G$ over $\Q$.   
\end{setup}

\subsubsection{Rational Goodwillie--Jones trace in degree zero.}\label{section: rational goodwillie-jones trace in degree zeo}
We start by discussing the following degree zero isomorphism.
\begin{theorem}[Rational Goodwillie--Jones trace in degree zero]\label{theorem: rational goodwillie--jones trace in degree zero}
For any $X \in \eB^G_{\Q}$, the trace induces a ring isomorphism
\begin{equation*}
    K^G_0(X; \Q) \xrightarrow{\cong} (HC^{-})^G_0(X / \Q).
\end{equation*}
\end{theorem}
\begin{proof}
The fiber $K^{\inf}(-; \Q) = \fib(K(-; \Q) \to HC^{-}(-/ \Q))$ is a truncating invariant \cite[proof of Corollary 3.9]{LT19}. We check that $K^{\inf}_0(BG; \Q) \cong 0 \cong K^{\inf}_{-1}(BG; \Q)$ in Lemma \ref{lemma: vanishing on BT} below.
We then apply Theorem \ref{theorem: a degree zero version} to $K^{\inf}(-; \Q)$ for $i = 0, -1$ separately. We deduce that for all $X \in \eB^G_k$ it holds that $K^{\inf}_0(X/G; \Q) \cong 0 \cong K^{\inf}_{-1}(X/G; \Q)$, so we get back the desired claim from the long exact sequence associated to the fiber sequence $K^{\inf}(-; \Q) \to K(-; \Q) \to HC^{-}(-/ \Q)$ on $X/G$.
\end{proof}

Again, we still need to supply the case of an equivariant point.
\begin{lemma}\label{lemma: vanishing on BT}\label{lemma: vanishing on BG}
We have $K^{\inf}_0(BG; \Q) \simeq 0 \simeq K^{\inf}_{-1}(BG; \Q)$.    
\end{lemma}
\begin{proof}
Since $HH(\pt/ \Q) \simeq \Q$ sitting in degree zero and $HH(-/ \Q)$ is finitary, we deduce from \eqref{equation: decomposition of BG} that $HH(BG/ \Q) \simeq R(G; \Q)$ is the rationalized representation ring sitting in degree zero. Now by the classical construction of $HC^{-}(-/ \Q)$ from $HH(-/ \Q)$ as in \cite[\S 5.1.7]{Lod97}, which is compatible with the categorical definition by \cite{Hoy18}, we see that $HC^{-}_{\bullet}(BG/ \Q) \cong HC^{-}_0(BG/ \Q)[u]$ as rings with the free generator $u$ sitting in degree $-2$. In other words, $HC^{-}(- / \Q)$ commutes with the direct sum decomposition \eqref{equation: decomposition of BG}. The same is true for $K(-; \Q)$, as it is finitary.

We display the following piece of the associated long exact sequence on $BG$:
\begin{equation*}
    HC^{-}_1(BG/ \Q) \to K^{\inf}_0(BG; \Q) \to K_0(BG; \Q) \to HC^{-}_0(BG/ \Q) \to K^{\inf}_{-1}(BG; \Q) \to K_{-1}(BG; \Q)
\end{equation*}
Now, $HC^{-}_1(BG/ \Q) \cong 0$ by above, the trace induces an isomorphism $K_0(BG; \Q) \cong R(G; \Q) \cong HC^{-}_0(BG / \Q)$, and $K_{-1}(BG; \Q) \cong 0$ by smoothness. Therefore $K^{\inf}_0(BG; \Q) \cong 0 \cong K^{\inf}_{-1}(BG; \Q)$ as desired.
\end{proof}

\begin{remark}
We also deduce the corresponding claim for the cdh-sheaffified theories from Appendix \ref{appendix: cdh sheaffified HH}: for all $X \in \eB^G_k$, we have an equivalence
\begin{equation*}
    KH^G_0(X; \Q) \xrightarrow{\cong} (\LHCn)^G_0(X/ \Q).
\end{equation*}
This is a weaker claim, where both sides are more computable.
\end{remark}

\begin{remark}
Under the geometric interpretation of equivariant negative cyclic homology via derived loop stacks \cite{BFN10, Chen20, HSS17, Toe14} using the fixed point scheme notation \cite[\S 1]{Low24}, we can rewrite Theorem \ref{theorem: rational goodwillie--jones trace in degree zero} as
\begin{equation*}
    K^G_0(X; \Q) \xrightarrow{\cong} \pi_0(\RG(\Fix^{\mathbf{L}}_{\frac{G}{G}}(X), \O)^{hS^1}).
\end{equation*}
In words, it gives an isomorphism between the zeroeth equivariant $K$-group of $X$ and homotopy-$S^1$-invariant global functions on the associated derived fixed-point scheme. 
\end{remark}

\subsubsection{A decomposition for positive degrees.}\label{section: a decomposition for positive degrees}
Let us complement Theorem \ref{theorem: rational goodwillie--jones trace in degree zero} with a slightly counterintuitive result on the positive homotopical degrees. In general, algebraic $K$-theory can be described by gluing the homotopy-invariant contribution of $KH(-)$ with the contribution of $\HCn(-/\Q)$ along their maps to the cdh-sheafification $\LHCn(-/\Q)$, which can be regarded as a truncating invariant. We recall this thoroughly in Appendix \ref{appendix: cdh sheaffified negative cyclic homology}.

It turns out that for simple varieties, the gluing condition is vacuous in positive degrees -- their positive $K$-groups decompose into the direct sum of $KH(-)$ and $HC^-(-/\Q)$.
\begin{proposition}\label{proposition: direct sum in positive degrees}
Let $X \in \eB^G_{\Q}$, then
\begin{equation*}
    K^G_i(X) \cong KH^G_i(X) \oplus (\HCn)^G_i(X/\Q), \qquad \forall i \geq 1
\end{equation*}
\end{proposition}
\begin{proof}
From the homotopy fiber square of Construction \ref{construction: cdh sheaffified localizing invariants} defining $\LHCn(-)$ it suffices to prove
\begin{equation*}
(\LHCn)^G_i(X/ \Q) \cong 0, \qquad  \forall i \geq 1   
\end{equation*}
This is the case for $X=\pt$ since $\LHCn_{\bullet}(BG / \Q) \cong HC^{-}_{\bullet}(BG / \Q) \cong HC^{-}_0(BG / \Q)[u]$ vanishes for all $i \geq 1$. We conclude by applying Theorem \ref{theorem: a degree zero version} to the $\Q$-linear truncating invariant $\LHCn(-/ \Q)$.
\end{proof}

\begin{remark}\label{remark: decomposition for toric varieties}
    In \cite[Proposition 5.6]{CHWW09}, a related phenomenon appears for non-equivariant $K$-theory of toric varieties.
\end{remark}

\subsection{Equivariant formality for homotopy invariant \texorpdfstring{$K$}{K}-theory}\label{section: homotopy invariant K-theory and equivariant formality}
We now move towards applications to homotopy invariant algebraic $K$-theory $KH$. We first record a strong equivariant formality statement in Theorem \ref{corollary: KH of simple varieties}. We then discuss two easy consequences: the comparison to its topological counterpart $K_{\top}$ over $\C$ in \S \ref{section: isomorphisms with topological theories in degree zero} and an elementary instance of Parshin vanishing over $\F_q$ in \S \ref{section: examples of the equivariant singular parshin property}.

\begin{setup}
Throughout this subsection, let $k$ be a base field and $G$ a nice group over it. Via Appendix \ref{appendix: actions of linearly reductive G}, the results generalize to any linearly reductive group $G$ over $k$.
\end{setup}

\subsubsection{Equivariant formality for \texorpdfstring{$KH$}{KH}.}\label{section: equivariant formality for KH}
Let us spell out the results from \S \ref{section: degreewise versions} for $KH$. This yields a strong equivariant formality statement for homotopy invariant $K$-theory of simple varieties. 
\begin{theorem}[$KH$ of simple varieties]\label{corollary: KH of simple varieties}
Let $k$ be any base field. For any $X \in \eB^G_k$ we have a natural ring isomorphism
\begin{equation*}
  KH^G_0(X) \underset{K_0(\pt)}{\otimes} K_{\bullet}(\pt) \xrightarrow{\cong}  KH^G_{\bullet}(X). 
\end{equation*}
Furthermore, $KH^G_0(X)$ is a finite projective module over the representation ring $KH^G_0(\pt) \cong R(G)$.
\end{theorem}
\begin{proof}
Follows from Theorem \ref{theorem: computation of truncating invariants} as $KH(-)$ is a finitary truncating invariant which is multiplicative. 
\end{proof}

\begin{remark}[Faithfull flatness]
Concisely, the above theorem shows that $KH^G(X)$ is connective and finite faithfully flat over $K^G(\pt)$.
\end{remark}

\begin{remark}[Change of group and equivariant formality]\label{remark: change of group and equivariant formality}
The equivalence of Theorem \ref{corollary: KH of simple varieties} is compatible with changing the group through any group homomorphism $G' \to G$. In particular, restricting to the trivial group $1 \to G$ specializes to non-equivariant $K$-theory by base change along the augmentation map
\begin{equation*}
    R(G) \to \Z.
\end{equation*}
This is an {\it equivariant formality} statement for $KH$ of varieties $X \in \eB^G_k$. It allows us to compute non-equivariant $KH$ directly from the equivariant one, where more techniques are available.
\end{remark}

\begin{remark}[Arbitrary base]\label{remark: KH arbitrary base}
Let $G$ be a group with decomposable representation theory over an arbitrary base ring $R$ (e.g. over $\Z$). The same arguments show that for any $X_R \in \Sch^G_R$ we have the relative statement
\begin{equation*}
  KH^G_0(X_R) \underset{KH_0(R)}{\otimes} KH_{\bullet}(R) \xrightarrow{\cong}  KH^G_{\bullet}(X_R)
\end{equation*}
functorially in base change along ring maps $R' \to R$. Indeed, the non-equivariant base case of $X_R = \Spec R$ is tautological, while the rest of our arguments goes through without change. 
\end{remark}

For the final remarks, let us focus on the case of a split torus $T = \Gm^n$ over $k$.
\begin{remark}[Freeness]
For any $X \in \eB^T_k$, we deduce that the rationalization $KH^T_0(X; \Q)$ is a finite free module over $\Q[t^{\pm 1}_1, \dots, t^{\pm 1}_n]$.

Indeed, it is a finite projective module over $KH^T_0(\pt; \Q) \cong \Q[t^{\pm 1}_1, \dots, t^{\pm 1}_n]$ by Corollary \ref{corollary: KH of simple varieties}. Any finite projective module over $\Q[t^{\pm 1}_1, \dots, t^{\pm 1}_n]$ is free by \cite{Swa78} or \cite[\S V.4, Corollaries V.4.10 and V.4.11]{Lam06}.
\end{remark}

\begin{remark}[Equivariant localization]\label{remark: equivariant formality of KH is useful for computations}
We expect Theorem \ref{corollary: KH of simple varieties} to be useful in conjunction with equivariant localization from \cite[Theorem C, \S 11]{KR21} for explicit computations of the rings $KH^T_{\bullet}(X)$ where $X \in \Sch^T_k$.

Indeed, given a $T$-equivariant variety $X$, equivariant localization postulates that the natural restriction map between $KH^T(X)$ and $KH^T(X^T)$ becomes an equivalence after localization on the base $K^T_0(\pt)$; the latter object is often easy to describe in terms of the $T$-equivariant geometry. Once equivariant formality holds, $KH^T(X)$ in particular embeds into its localization. The computational question then becomes to describe its image.

For cohomology, the above recipe is well-known and useful. For $KH$ of singular varieties, not much has been done. For example, to the best of our knowledge, there are basically no results on equivariant formality of $KH$ for singular varieties in the literature to start with.
\end{remark}

\subsubsection{Isomorphisms with topological theories in degree zero.}\label{section: isomorphisms with topological theories in degree zero}
If we work over the complex numbers $k = \C$, we can compare to topological $K$-theory. More precisely, we have maps of localizing invariants
\begin{equation*}
K(-) \to KH(-) \to K_{\st}(-) \to K_{\top}(-)
\end{equation*}
of $\C$-linear finitary localizing invariants \cite{Bla16}. Moreover, $KH(-)$, $K_{\st}(-)$, $K_{\top}(-)$ are truncating.

These three latter theories are believed to be reasonably close to each other, but again examples do not come in plenty. They are indeed close for simple varieties: their degree zero parts\footnote{To avoid any potential confusion, we emphasize that in our conventions, $K_{\top}(X)$ matches the topological {\it $K$-cohomology } built up from vector bundles. In particular, $K_{\top, 0}(X)$ matches the zeroeth topological $K$-cohomology group (which is usually denoted with a superscript).} are actually isomorphic (while the rest is determined by their value on the point).
\begin{proposition}[$KH$ and $K_{\top}$ of complex varieties]\label{theorem: isomorphisms with topological theories in degree zero}
Let $X \in \eB^G_{\C}$, then the natural maps induce ring isomorphisms
\begin{equation*}
  KH^G_0(X) \to (K_{\st})^G_0(X) \to (K_{\top})^G_0(X). 
\end{equation*}
\end{proposition}
\begin{proof}
This is clearly the case for the non-equivariant point $\pt = \Spec \C$, where all three terms take the value $\Z$. We conclude by Theorem \ref{theorem: a degree zero version} and Lemma \ref{lemma: value on equivariant point}.   
\end{proof}

\begin{remark}[Equivariant formality]
Note that Theorem \ref{theorem: computation of truncating invariants} applies equally well to the truncating finitary multiplicative invariants $K_{\st}(-)$, $K_{\top}(-)$ as to $KH(-)$ in \S \ref{section: homotopy invariant K-theory and equivariant formality}. All of them are equivariantly formal for any $X \in \eB^G_{\C}$, compatibly with the comparison maps. Similarly for the non-finitary truncating invariant $HP(-/\C)$.
\end{remark}

\begin{remark}[Topological and arithmetic parts of $KH$]
Paired with the equivariant formality from previous remark, Proposition \ref{theorem: isomorphisms with topological theories in degree zero} allows to compute $KH^G_{\bullet}(X)$ by tensoring two independent contributions: the topological contribution of $K^G_{\top, 0}(X)$ and the arithmetic contribution of $K_{\bullet}(\C)$.
\end{remark}

To free the careful reader of any potential doubts, we would like to emphasize that $K^G_{\top}(-)$ is the classical invariant one hopes for.
\begin{remark}[Blanc vs. Segal]
The equivariant topological $K$-theory $K^G_{\top}$ agrees with the classical construction \cite{Seg68} on any $X \in \Sch^G_{\C}$. 

Denote $M \subseteq G(\C)$ the maximal compact subgroup of the topological group of complex points of $G$ with the analytic topology. Then $M$ is a compact real Lie group and it acts on the topological space $X(\C)$ of complex points of $X$ with the analytic topology. In this situation, \cite{Seg68} considered the equivariant topological $K$-theory spectrum
\begin{equation*}
    K^M_{\top}(X(\C)).
\end{equation*}
For $X$ in $\Sch^G_{\C}$, we have a natural comparison map
\begin{equation*}
    K^G_{\top}(X) \xrightarrow{} K^M_{\top}(X(\C)).
\end{equation*}
constructed in \cite[\S 2.1.3]{HLP20}. 

Now, both sides have proper excision: for $K^G_{\top}$ this holds because it is a truncating invariant via \cite{ES21, LS25}, while for $K^M_{\top}$ this follows from the long exact sequence of a pair. By equivariant resolution of singularities in characteristic zero (see \cite[\S 6.1]{EKS25} for a discussion), it is thus enough to check that the comparison map is an equivalence on smooth $G$-equivariant schemes. This has been done in \cite{HLP20}, \cite{Bla16}. 
\end{remark}

For the sake of completeness, let us mention the following fact: the topological realization of the trace map is always an isomorphism.
\begin{remark}[Topological Chern character]
On topological $K$-theory, the trace map induces the topological Chern character
\begin{equation*}
    K_{\top}(-) \otimes \C \to HP(-/\C).
\end{equation*}
For any $X \in \Sch^G_{\C}$ we have an equivalence
\begin{equation*}
K^G_{\top}(X) \otimes \C \xrightarrow{\simeq} HP^G(X/\C).    
\end{equation*}
For smooth schemes, this is due to \cite{Kon21}. It was extended to smooth quotient stacks by \cite{HLP20}. Also see \cite[Theorem 5.3.2]{ES21}. For possibly singular stacks with nice stabilizers, this follows by cdh descend; see \cite{Kha23}. This works for quotients of quasi-projective varieties by linearly reductive $G$ via Appendix \ref{appendix: actions of linearly reductive G}.
\end{remark}

\subsubsection{Examples of the equivariant singular Parshin property.}\label{section: examples of the equivariant singular parshin property}
We also mention the following silly consequence of the formality of $KH$ from Theorem \ref{corollary: KH of simple varieties}.
If $k = \F_q$ is a finite field of characteristic $p$ with $q$ elements, each $X \in \eB_{\F_q}$ satisfies Parshin's property:
\begin{equation*}
K_i(X; \Q) \cong 0, \qquad \forall i \geq 1.   
\end{equation*}
This property is usually conjectured for smooth projective varieties over $\F_q$ and such claims are open in general. It is certainly expected for singular projective varieties over $\F_q$ -- the statement is compatible with cdh descent, so the singular case would follow from the smooth case if one had resolution of singularities for projective varieties over $\F_q$.

In fact, this argument works even equivariantly.
\begin{observation}\label{observation: stability of the Parshin property}
Let $G$ be a nice group over $\F_q$ and $(X, Y, Z, E)$ be an abstract blowup square in $\Sch^G_{\F_q}$. If the Parshin's property 
\begin{equation*}\label{equation: equivariant parshin conjecture}
   K^G_i(-; \Q) \cong 0, \qquad \forall i \geq 0
\end{equation*}
holds for $Y, Z, E$, then it holds for $X$ as well.
\end{observation}
\begin{proof}
Since $K^G(X; \Q) \simeq KH^G(X; \Q)$ by \cite[Theorem 1.3.(2)]{KR18} and \cite[Theorem 1.3.(4)]{Hoy21}, it satisfies cdh descent, so we conclude from the associated long exact sequence.    
\end{proof}

However, not many instances of this method are available in practice. The equivariant formality of $KH$ from Theorem \ref{corollary: KH of simple varieties} now easily implies that varieties from $\eB^G_k$ qualify. 

\begin{example}\label{example: parshins property for simple varieties}
Let $G$ be a nice group over $\F_p$. For any $X \in \eB^G_{\F_q}$, we have
\begin{equation*}
   K^G_i(X; \Q) \cong 0, \qquad \forall i \neq 0.
\end{equation*}
In particular, $X$ satisfies Parshin's property.
\end{example}
\begin{proof}
To this end, compute
\begin{align*}
K^G_{\bullet}(X; \Q) 
\cong KH^G_{\bullet}(X; \Q)
\cong KH^G_0(X) \underset{K_0(\pt)}{\otimes} K_{\bullet}(\pt) \underset{\Z}{\otimes} \Q
\cong KH^G_0(X; \Q).
\end{align*} 
Here, the first isomorphism holds by \cite[Theorem 1.3.(2)]{KR18} and \cite[Theorem 1.3.(4)]{Hoy21} already with $\Z[\tfrac{1}{p}]$-coefficients; in the non-equivariant case this goes back to \cite{Wei89, TT90}. The second isomorphism comes from Theorem \ref{corollary: KH of simple varieties} and the final isomorphism from the knowledge that $K_\bullet(\pt; \Q)$ is concentrated in degree zero with value $\Q$ by Quillen's work \cite[Corollary IV.1.13]{Weib13}.
\end{proof}

On the other hand, we can illustrate Observation \ref{observation: stability of the Parshin property} on the following example, which is not directly covered by the class $\eB^G_{\F_q}$ (although see \S \ref{section: outlook}). Recall \cite{VV03} for a quick overwiew of split toric varieties over $k$. 
\begin{example}\label{example: parshin property for toric varieties}
Split toric varieties over $\F_p$ satisfy the equivariant Parshin's property.
\end{example}
\begin{proof}
Note that this is the case for the point; it then holds for smooth split toric varieties by \cite[Theorem 6.2]{VV03}. Since the statement is compatible with enlarging $T$ along precomposition with a group homomorphism by Remark \ref{remark: classes of simple varieties for different groups}, the case of singular split toric varieties reduces to the smooth case via toric resolutions and cdh descent from Observation \ref{observation: stability of the Parshin property}.
\end{proof}

\begin{remark}\label{remark: broken toric varieties}
The Parshin's property also holds for the so called \textit{broken toric varieties}, defined by closed gluing of toric varieties along toric subvarieties. Indeed, the statement is clear from Example \ref{example: parshin property for toric varieties} via Observation \ref{observation: stability of the Parshin property} -- see Example \ref{example: closed gluing} below for the relevant abstract blowups. Broken toric varieties (with smooth irreducible components) naturally appear in the mirror symmetry literature, see \cite{Sun24} and references there.
\end{remark}

\section{Examples of simple varieties}\label{section: examples of simple varieties}
We now pay our debt to the reader by providing examples of varieties for which the above theory applies. Although such varieties are rather special, we recover examples of interest in geometric representation theory. After some basic examples in \S \ref{section: some basic examples}, we discuss both finite and affine Schubert varieties \S \ref{section: classical scubert varieties}, \S \ref{section: affine schubert varieties} for $\GL_n$. The singularities of such varieties are of deep interest in geometric representation theory.

\subsection{Some basic examples}\label{section: some basic examples}
We start with a few elementary examples. Let $k$ be a base field and $G$ a linearly reductive group over it.

\begin{example}[Partial flag varieties]\label{example: partial flag varieties}
Projective spaces $\P^n$, Grassmannians $\mathrm{Gr}(n, d)$ and all partial flag varieties $\Flag(n, d_{\bullet})$ for $GL_n$ lie in $\eB_k$ and hence in $\eC_k$ by \eqref{closure property: the point} and \eqref{closure property: projective bundles}. Picking any $G$-action on the underlying vector space $V$, these $G$-equivariant varieties lie in $\eB^G_k$ and hence in $\eC^G_k$.     
\end{example}

\begin{example}[Hirzerbruch surfaces]
Let $m, d \geq 1$. Consider a direct sum of line bundles $\eE = \O(i_1) \oplus \dots \O(i_{m})$ on $\P^d_k$ and let $X := \P_{\P^d}(\eE)$ be its projectivization. All of this further carries an obvious action of the torus $T = \Gm^{d+1}$.
Then $X \in \eB^T_k$ by \eqref{closure property: the point} and \eqref{closure property: projective bundles}.

It is known that this construction yields all smooth projective toric varieties of Picard rank $2$. 
For example, specializing to $\eE = \O \oplus \O(-n)$ on $\P^1$ yields Hirzerbruch surfaces $\Sigma_n$.
\end{example}

\begin{example}[Cusp]\label{example: projective cuspidal curve}
Let $X$ be the projective cuspidal curve over $k$. Then $X \in \eC_k$. Indeed, this follows from the abstract blowup square    
\begin{equation*}
\begin{tikzcd}
        \P^1 \arrow[d] \arrow[r, hookleftarrow] &  \arrow[d] \Spec ( k[\varepsilon]/(\varepsilon^2)) \\
        X \arrow[r, hookleftarrow] & \pt
\end{tikzcd}
\end{equation*}
whose other terms lie in $\eC_k$ by \eqref{closure property: the point}, \eqref{closure property: projective bundles} and Remark \ref{remark: independence on nilpotent structures}. We conclude via \eqref{closure property: 3-out-of-4 for abstract blowups}.
\end{example}

\begin{example}[Node]\label{example: projective nodal curve}
Let $X$ be the projective nodal curve over $k$. Then $X \in \eC_k$, but not in $\eB_k$. Indeed, there is an abstract blowup square
\begin{equation*}
\begin{tikzcd}
        \P^1 \arrow[d] \arrow[r, hookleftarrow] &  \arrow[d] \pt \sqcup \pt \\
        X \arrow[r, hookleftarrow] & \pt
\end{tikzcd}
\end{equation*}
whose other terms lie in $\eC_k$ by \eqref{closure property: the point} and \eqref{closure property: projective bundles}.
Hence also $X$ lies in $\eC_k$ by Definition \ref{definition: simple varieties C}.  

On the other hand, if $X$ was in $\eB_k$, then we would also have $KH_{-1}(X) \cong 0$ by Theorem \ref{corollary: KH of simple varieties}. However, the long exact sequence associated with the abstract blowup square above ends with
\begin{equation*}
    \begin{tikzcd}
      \arrow[r] & KH_0(\P^1) \oplus KH_0(\pt) \arrow[r, "\rank"] \arrow[d, equal] & KH_0(\pt \sqcup \pt) \arrow[r] \arrow[d, equal] &  KH_{-1}(C) \arrow[r] \arrow[d, equal] & 0  \\
      \arrow[r] & \Z^{\oplus 3} \arrow[r, "(a{,} b{,} c) \mapsto (a+b+c{,} a+b+c)"] & \Z^{\oplus 2} \arrow[r] &  KH_{-1}(C) \arrow[r] & 0 
    \end{tikzcd}
\end{equation*}
showing $KH_{-1}(X) \cong \Z$. Altogether, $X \notin \eB_k$.
\end{example}

\begin{example}[Closed gluing]\label{example: closed gluing}
If $X$ is covered by two closed subvarieties $Z_1$ and $Z_2$ with (reduced) intersection $Z_3 := Z_1 \cap Z_2$ such that $Z_1, Z_2, Z_3 \in \eC_k$, then also $X \in \eC_k$. Similarly for the equivariant case. Indeed, this immediately follows from the associated abstract blowup square featuring $(X, Z_1 \sqcup Z_2, Z_3, Z_3 \sqcup Z_3)$ and from Definition \ref{definition: simple varieties C}.
\end{example}

\begin{example}[Closure on projective cones]\label{example: closure on projective cones}
Let $k$ be any base field -- or even any base ring -- and $E$ a projective scheme over $k$. Assume that $E \in \eB_k$ (resp. $\eC_k$). Let $i: E \hookrightarrow \P^n_k$ be any projective embedding, $\eL = i^* \O(1)$ the corresponding line bundle on $E$, and $X$ the associated projective cone. Then $X \in \eB_k$ (resp. $\eC_k$) as well. Similarly for the equivariant case. 

Indeed, blowing up $X$ in the cone point, we obtain an abstract blowup
\begin{equation*}
\begin{tikzcd}
        Y \arrow[bend left=30, r, dashed] \arrow[d] \arrow[r, hookleftarrow, swap, "0"] &  \arrow[d]  E \\
        X \arrow[r, hookleftarrow] & \pt
\end{tikzcd}
\end{equation*}
where $Y = \P_E(\O \oplus \eL)$ and the upper horizontal map is the inclusion of the zero section, which is split by the structure map $\P_E(\O \oplus \eL) \to E$. 

The above statement is covered by \cite[\S 8]{EGAII}: under the notation from \cite[(8.3.1) and (8.3.2)]{EGAII}, the relevant commutative square is \cite[(8.7.1)]{EGAII}. It is indeed a pullback square by \cite[(8.7.8)]{EGAII}; the description of $Y$ and the horizontal inclusion follows from \cite[(8.7.8) and (8.4.2)]{EGAII}. Alternatively see \cite[(8.8.2), (8.8.3), (8.8.4)]{EGAII} for discussion of the relevant affine chart away from infinity.

Altogether, we have $E \in \eB_k$ by assumption, hence $Y = \P_E(\O \oplus \eL) \in \eB_k$ by \eqref{closure property: projective bundles}. Also $\pt \in \eB_k$ by \eqref{closure property: the point}. Since the abstract blowup square is split, we deduce that $X \in \eB_k$ by \eqref{closure property: 3-out-of-4 for split abstract blowups}. The discussion for the class $\eC_k$ is similar; the equivariant versions hold by naturality of the arguments.
\end{example}

\begin{example}[A projective cone of $\P^1$]\label{example: the projective cone of projective line}
Let $k = \Q$. To illustrate that the decomposition from Proposition \ref{proposition: direct sum in positive degrees} can be nontrivial, we recall the example of the projective cone $X$ of $\P^1$ under the embedding by $\O(2)$, heavily based on the results of \cite{CHWW13} for the affine cone. Note that $X$ accidentally matches the affine Schubert variety $X_{\leq 2\omega_1} \hookrightarrow \Gr_{\GL_2}$ by \cite[Lemma 2.1.14]{Zhu15}. Also compare to \cite{PS21}.

\noindent We claim that
\begin{align*}
K_i(X) =
\begin{cases}
KH_i(X) \oplus \HCn_i(X/ \Q) & \text{if } i \geq 1 \\
KH_0(X) & \text{if } i=0 \\
0 & \text{if } i \leq -1
\end{cases}
\end{align*}
where
\begin{equation*}
KH_i(X) \cong KH_i(\Q)^{\oplus 3} \ \ \ \text{if } i \geq 0
\qquad \text{and} \qquad
\HCn_i(X / \Q) \cong 
\begin{cases}
0 & \text{if } i \geq 1 \text{ even,} \\
\Q & \text{if } i \geq 1 \text{ odd.}
\end{cases}
\end{equation*}

To get this, write $Y = \P_{\P^1}(\O \oplus \O(2))$ for the resolution of $X$. This matches the affine Demazure resolution $Y_{\leq (\omega_1, \omega_1)}$, but we do not need this fact. We stratify $X$ into the affine cone $X_1$ and the complement of the singular point $X_2$, and write $X_3 = X_1 \cap X_2$ for their intersection. Pulling back this stratification to $Y$, we get the diagram
\begin{equation*}
    \begin{tikzcd}
        Y \arrow[d, "f", swap] & \arrow[l] Y_1 \sqcup Y_2 \arrow[d, "f_1 \sqcup f_2", swap] & \arrow[l, shift left=0.75ex] \arrow[l, shift right=0.75ex, swap] \arrow[d, "f_3", swap] Y_3 \\
        X & \arrow[l] X_1 \sqcup X_2 & \arrow[l, shift left=0.75ex] \arrow[l, shift right=0.75ex, swap] X_3 
    \end{tikzcd}
\end{equation*}
where rows are Zariski covers. The maps $Y_2 \to X_2$ and $Y_3 \to X_3$ are isomorphisms. We now apply $K(-)$ to this diagram; rows will then induce long exact sequences by Zariski descent. 
\begin{equation*}
    \begin{tikzcd}
        K(Y) \arrow[r] & K(Y_1) \oplus K(Y_2) \arrow[r] & K(Y_3) \\
        K(X) \arrow[u, "f^*"] \arrow[r] & K(X_1) \oplus K(X_2) \arrow[u, "f^*_1 \oplus \id"] \arrow[r] & K(X_3) \arrow[u, "\id"]
    \end{tikzcd}
\end{equation*}
The upper row consists of smooth varieties, and we have $K(Y) = K(\Q)^{\oplus 4}$ by repeated use of the projective bundle formula. From the split abstract blowup square associated to the resolution $f: Y \to X$, we deduce $KH(X) = K(\Q)^{\oplus 3}$ is a direct summand in $KH(Y) = K(Y)$. Since $X_2$ and $X_3$ are smooth, the remaining difference between $K(X)$ and $K(Y)$ comes from $K(X_1)$, which is described by \cite[Theorem 4.3 and Remark 4.3.1]{CHWW13} -- it gives $\Q$ in odd positive degrees $2j +1 \geq 1$ and zero else. However, each of these $\Q$ maps to $0$ under $f^*_1: K_{2j+1}(X_1) \to K_{2j+1}(Y_1)$ -- indeed, $\Q$ is uniquely divisible, while $K_{2j+1}(Y_1) \cong K_{2j+1}(\Q)^{\oplus 2}$ is torsion for $2j+1 \geq 3$ by \cite{Weib13} and $(\Q^{\times})^{\oplus 2}$ for $2j+1 = 1$ (alternatively one can keep track of the weights). From the commutativity the right-hand square and the fact that the right vertical map is the identity, we now see that the whole $K_{2j+1}(X_1) \cong \Q$ maps to zero under $K_{2j+1}(X_1) \to K_{2j+1}(X_3)$ -- hence it contributes to $K_{2j+1}(X)$ by long exact sequence induced by the bottom row. We have thus matched $K_i(X) = KH_i(X) = KH(\Q)^{\oplus 3}$ when $i$ is not odd positive; there is an extra $\Q$ appearing if $i$ is odd positive.

On the other hand, we already know that $K_i(X) \cong KH_i(X) \oplus \HCn_i(X / \Q)$ for all $i \geq 1$ by Proposition \ref{proposition: direct sum in positive degrees}. Altogether, we deduce the result.
\end{example}

\begin{remark}
Consider an affine Demazure resolution $f: Y \to X$ over $k$. Since $X$ has rational singularities, the pullback $f^*:\Perf(X) \to \Perf(Y)$ is a fully faithful embedding. One could wishfully ask whether $f^*$ makes $\Perf(X)$ into a semi-orthogonal summand of $\Perf(Y)$. This is not the case.

In fact, it fails by the following general principle. We work over a field $k$. Under mild assumptions, regularity of a scheme $Z$ is equivalent to the existence of a \textit{strong generator} of $\Perf(Z)$ via \cite[Theorem 3.27]{Orl16}. Furthermore, strong generators are inherited by semi-orthogonal summands \cite[Proposition 3.20]{Orl16}. Put together, if $\Perf(X)$ is a semi-orthogonal summand of $\Perf(Y)$ for a regular scheme $Y$, then $X$ is regular as well.

In the case of affine Schubert varieties, the above failure can be seen already on the level of localizing invariants. Consider the map $f: Y \to X$ over $\Q$ from Example \ref{example: the projective cone of projective line}, which accidentally matches the affine Demazure resolution $f: Y_{\leq (\omega_1, \omega_1)} \to X_{\leq 2\omega_1}$ of the affine Schubert variety $X_{\leq 2\omega_1} \hookrightarrow \Gr_{\GL_2}$ by \cite[Lemma 2.1.14]{Zhu15}. 
If $\Perf(X)$ were a semi-orthogonal summand of $\Perf(Y)$, then the natural map $E_{\bullet}(X) \to E_{\bullet}(Y)$ would be injective for any localizing invariant $E$. This is wrong for negative cyclic homology $HC^{-}(-/\Q)$. By the case of a point and projective bundle formula, $HC^{-}(Y/\Q)$ is supported in non-positive degrees, but $HC^{-}(X/\Q)$ is nonzero in each odd positive degree (it is a copy of $\Q$) by Example \ref{example: the projective cone of projective line}.    
\end{remark}

\subsection{Finite Schubert varieties}\label{section: classical scubert varieties}
We prove that finite Schubert varieties in usual Grassmannians are simple varieties in Lemma \ref{lemma: finite schubert varieties are simple} -- they lie in the class $\eB^B_k$, and in particular in $\eB^T_k$.
Already here, our computations seem to be new. 

We start by quickly recalling the definitions; see \cite{Bri04} or \cite[\S 2.2, \S 2.3]{Oet21} for these standard facts. 
For the sake of readability, we use the convention dual to Notation \ref{notation: flag bundles}: given $X \in \Sch_k$, $\eF \in \Qcoh(X)$ and a dimension vector $d_{\bullet} = (d_1, \dots, d_m)$, we denote $\Flag^*_X(\eF, d_{\bullet})$ the scheme parametrizing flags of {\it subbundles} -- as opposed to quotient bundles -- with successive subquotients of ranks $d_1, \dots, d_m$. In particular, $\Grass^*(\eF, d)$ parametrizes rank $d$ locally free {\it subbundles}, compatibly with the above references.

\begin{setup}\label{setup: finite schubert varieties}
Let $k$ be a base field and $U$ a $k$-vector space of dimension $n$. Let $F_{\bullet}U = (0 = F_0U \subseteq F_1U \subseteq F_2U \subseteq \dots \subseteq F_{n}U = U)$ be a fixed full reference flag in $U$. We denote by $B \leq GL_n$ the Borel subgroup stabilizing $F_{\bullet}U$ and by $T = \Gm^{n}$ its maximal torus. We choose an embedding $T \leq B$.
\end{setup}

\begin{recollection}[Finite Schubert varieties]
Fix a sequence of non-negative integers $j_{\bullet} = (0=j_0 \leq j_1 \leq \dots \leq j_n=d)$ with $j_i \leq i$ for each $i = 0, \dots, n$. Denote $d_{\bullet} = (d_1, \dots, d_n)$ the corresponding difference sequence with $d_i := j_i - j_{i-1}$ for each $i = 1, \dots, n$. The associated {\it finite Schubert variety} $X_{\leq j_{\bullet}}$ is the reduced closed subvariety of $\Grass^*_k(U,d)$ classifying
\begin{equation*}
   X_{\leq j_{\bullet}} = \{ V \subseteq U \mid \dim_k V = d, \ \forall i: \ \dim_k(V \cap F_iU) \geq j_i  \}.
\end{equation*}
It is a projective variety with a natural action of $T \leq B$. Set-theoretically, it is a closed union of $B$-orbits in $\Grass^*_k(U,d)$. It is often singular; the smooth cases are determined by explicit combinatorial criteria.
\end{recollection}

\begin{lemma}[Finite Schubert varieties are simple]\label{lemma: finite schubert varieties are simple}
For any $d$ and $j_{\bullet}$ as above, we have $X_{\leq j_{\bullet}} \in \eB^B_k$. In particular, $X_{\leq j_{\bullet}} \in \eB^T_k$.
\end{lemma}

\begin{proof}
Consider the Bott--Samelson resolution $Y_{\leq j_{\bullet}}$, which classifies
\begin{equation*}
    Y_{\leq j_{\bullet}} = 
    \{ (0 = V_0 \subseteq V_1 \subseteq V_2 \subseteq \dots \subseteq V_{n}) \mid V_i \subseteq F_iU, \ \forall i: \ \dim_k V_i = j_i  \}
\end{equation*}
The map $f: Y_{\leq j_{\bullet}} \to X_{\leq j_{\bullet}}$ given by $(V_0 \subseteq V_1 \subseteq V_2 \subseteq \dots \subseteq V_{n}) \mapsto V_n$ is proper birational and $B$-equivariant.

On one hand, $Y_{\leq j_{\bullet}}$ is a tower of honest Grassmannian bundles over $\pt$, defined inductively as follows: in the $i$-th step, we consider the Grassmannian classifying $d_i$-dimensional subspaces in $F_iU$ modulo the tautological bundle with fiber $V_{i-1}$ from the previous step. 

On the other hand, the map $f: Y_{\leq j_{\bullet}} \to X_{\leq j_{\bullet}}$ can be constructed as a tower of stratified Grassmannian bundles, defined inductively as follows. Denote $\eV$ the tautological bundle on $X_{\leq j_{\bullet}}$ with fiber $V$. In the $i$-th step, we consider the Grassmannian classifing $d_i$-dimensional subspaces in the pullback of $\eV \cap F_{\bullet }U_i$ modulo the tautological subbundle from the previous step (also note that by definition of $X_{\leq j_{\bullet}}$ and $\eV$, this coherent sheaf has rank $\geq d_i$ everywhere).

Altogether, we get a $B$-equivariant diagram
\begin{equation*}
    X_{\leq j_{\bullet}} \xleftarrow{f} Y_{\leq j_{\bullet}} \xrightarrow{g} \pt
\end{equation*}
where $f$ is an iterated stratified Grassmannian bundle with nonempty fibers and $g$ is an honest iterated Grassmannian bundle. Since $\pt \in \eB^B_k$ by \eqref{closure property: the point}, also $Y_{\leq j_{\bullet}} \in \eB^B_k$ by iterative use of \eqref{closure property: projective bundles}. Hence $X_{\leq j_{\bullet}} \in \eB^B_k$ by iterative use of \eqref{closure property: stratified projective bundles}, the relevant rank condition being satisfied by the above. In particular, $X_{\leq j_{\bullet}} \in \eB^T_k$ by Remark \ref{remark: classes of simple varieties for different groups}.
\end{proof}

\subsection{Affine Schubert varieties}\label{section: affine schubert varieties}
Let $k$ be any base field. Consider the split group $GL_n$ over $k$ with its diagonal torus $T$. The motivating examples of our interest are the affine Schubert varieties in the $GL_n$ affine Grassmannian $\Gr$. These are certain equivariant singular projective varieties; they are of considerable interest in geometric representation theory. We show in Lemma \ref{example: affine schubert varieties are simple} that they lie in $\eB^{\GL_n}_k$, and in particular in $\eB^T_k$.

Before going to the proof, we quickly recall the relevant geometric input. See \cite{Zhu15} for details. Also \cite{BS16} is helpful (although the geometric setup is different, their description of the Demazure morphism is applicable and very clear).

\begin{recollection}[Affine Schubert varieties]
 The $\GL_n$ affine Grassmannian \cite[\S 1.1]{Zhu15} is the ind-scheme representing the moduli problem 
\begin{equation*}
   \Gr: R \mapsto \{ \Lambda \subseteq R\llp t \rrp^{\oplus n} \mid \Lambda \text{ full } R\llb t \rrb\text{-lattice} \}.
\end{equation*}
Let $\Lambda_0 = R\llb t \rrb^{\oplus n} \subseteq R\llp t \rrp^{\oplus n}$ be the standard lattice. Given any dominant coweight $\mu \in X^+_{\bullet}(T)$, we obtain the corresponding affine Schubert variety $X_{\leq \mu}$. By definition \cite[\S 2.1]{Zhu15}, this is the reduced subscheme of $\Gr$ on the closed subfunctor
\begin{equation*}
  R \mapsto \{ \Lambda \subseteq R\llp t \rrp^{\oplus n} \mid \Lambda \text{ full } R\llb t \rrb\text{-lattice}, \text{ the relative position of } \Lambda \text{ to } \Lambda_0 \text{ is } \leq \mu\}.
\end{equation*}
It is a projective variety over $k$ with an action of the positive loop group $\lop GL_n$. Hence it carries an action of $GL_n$ and in particular of $T$.
\end{recollection}

\begin{recollection}[The tautological bundle]
The variety $X_{\leq \mu}$ carries a natural vector bundle $\eE = \eE_{\leq \mu}$ defined as follows (see also \cite[\S 1.1]{Zhu15} and \cite[\S 7]{BS16}). We write $\mo^{\fl}(R\llb t \rrb)$ for the abelian category of finite length $R\llb t \rrb$-modules. Without loss of generality, assume that we can write $\mu = \mu_1 + \dots + \mu_{\ell}$ as a sum of $\ell$ fundamental dominant coweights $\mu_i \in X^{+}_{\bullet}(T)$ and call $\ell = \lg(\mu)$ the length of $\mu$.  Under this assumption, any $\Lambda \in X_{\leq \mu}(R)$ lies inside $\Lambda_0$. (The assumption is indeed harmless -- all $\mu \in X^{+}_{\bullet}(T)$ are of this form up to adding a negative multiple $-m \omega_n$ of the determinant coweight, i.e. replacing $\Lambda_0$ by $t^{-m}\Lambda_0$.) We then define
\begin{align*}
  \eE: X_{\leq \mu}(R) & \to \mo^{\fl}(R\llb t \rrb) \\
     \Lambda & \mapsto \Lambda_0 / \Lambda.
\end{align*}
This gives a vector bundle on $X_{\leq \mu}$, carrying a nilpotent operator $t: \eE \to \eE$. It is naturally $\lop GL_n$-equivariant. Consider now
\begin{equation*}
  \eE_{\bullet} := [ \eE \xrightarrow{t} \eE ] \ \qquad \text{and} \qquad \eF := \eH_0(\eE_{\bullet}) = \coker ( \eE \xrightarrow{t} \eE ).
\end{equation*}
Then $\eE_{\bullet}$ is a perfect complex of Tor amplitude $[1, 0]$ resolving the coherent sheaf $\eF$.
\end{recollection}

\begin{recollection}[Affine Demazure resolutions]
Given any sequence $\mu_{\bullet} = (\mu_1, \dots, \mu_{\ell})$ summing up to $\mu$, we get a corresponding (partial) affine Demazure resolution $Y_{\leq \mu_{\bullet}}$, see \cite[(2.1.17)]{Zhu15}. This is the underlying reduced scheme of the moduli problem
\begin{equation*}
  R \mapsto \{ (\Lambda_{\ell}, \dots, \Lambda_0) \mid \Lambda \subseteq R\llp t \rrp^{\oplus n} \text{ full } R\llb t \rrb\text{-lattice}, \text{ the relative position of } \Lambda_i \text{ to } \Lambda_{i-1} \text{ is } \leq \mu_{i}\}.
\end{equation*}
The {\it convolution map} $f: Y_{\leq \mu_{\bullet}} \to X_{\leq \mu}$ given by $(\Lambda_{\ell}, \dots, \Lambda_0) \mapsto \Lambda_{\ell}$ is equivariant, proper, birational.

In particular, write $\mu = \mu_1 + \lambda$ with $\mu_1$ minuscule and $\lg(\lambda) < \lg(\mu)$. Then the convolution map $f$ is given by the structure map of the stratified Grassmannian bundle
\begin{equation*}
    Y_{\leq (\mu_1, \lambda)} = \Grass_{X_{\leq \mu}}(\eF, \mu_1) \xrightarrow{f} X_{\leq \mu}.
\end{equation*}
On the other hand, since $\mu_1$ is minuscule, $Y_{\leq \mu_{\bullet}}$ is an honest Grassmannian bundle $\Grass_{X_{\leq \lambda}}(V, \mu_1)$ over $X_{\leq \lambda}$ associated to an $n$-dimensional vector space $V$ with a $\lop GL_n$-action.

Altogether, we obtain the convolution diagram
\begin{equation}\label{diagram: partial demazure resolution as stratified grassmannian bundle}
   X_{\leq \mu} \xleftarrow{f} Y_{\leq \mu_{\bullet}} \xrightarrow{g} X_{\leq \lambda}
\end{equation}
which is $\lop GL_n$-equivariant, and in particular $GL_n$-equivariant.
\end{recollection}

We are now ready to deduce that affine Schubert varieties are simple.
\begin{lemma}[Affine Schubert varieties are simple]\label{example: affine schubert varieties are simple}
Let $X_{\leq \mu}$ be any affine Schubert variety in the $GL_n$ affine Grassmannian, then 
$$X_{\leq \mu} \in \eB^{\GL_n}_k.$$
In particular, $X_{\leq \mu} \in \eB^{T}_k$.
\end{lemma}

\begin{proof}
We will prove that $X_{\leq \mu} \in \eB^{\GL_n}_k$ by induction on the length of $\mu$.

If $\lg(\mu) \leq 1$, then $X_{\leq \mu}$ is either the point $\pt$ or an honest Grassmannian $\Grass_k(V, \mu_i)$ associated to an $n$-dimensional vector space $V$ with an action of $GL_n$. We thus have $X_{\leq \mu} \in \eB^{\GL_n}_k$ in this case by the closure properties \eqref{closure property: the point} and \eqref{closure property: projective bundles}.

Now assume $\lg(\mu) \geq 2$ and write $\mu = \mu_i + \lambda$ with $\mu_i$ minuscule and $\lg(\lambda) < \lg(\mu)$. We get the convolution diagram \eqref{diagram: partial demazure resolution as stratified grassmannian bundle}

By induction, we already know that $X_{\leq \lambda} \in \eB_k^{\GL_n}$. Since $Y_{\leq \mu_{\bullet}}$ is an honest Grassmannian bundle over it, we deduce $Y_{\leq \mu_{\bullet}} \in \eB^{\GL_n}_k$ by the closure property \eqref{closure property: projective bundles}. But now $Y_{\leq \mu}$ is a stratified Grassmannian bundle over $X_{\leq \mu}$, so we deduce $X_{\leq \mu} \in \eB^{\GL_n}_k$ by the closure property \eqref{closure property: stratified projective bundles}. This completes the inductive step.

In particular, $X$ lies in $\eB^T_k$ and $\eB_k$ by Remark \ref{remark: classes of simple varieties for different groups}
\end{proof}

\begin{remark}
The same argument goes through with respect to the groups $T \times \Gm \leq GL_n \times \Gm$ extended by the loop rotation action of $\Gm$. 
\end{remark}

\appendix
\section{Cdh-sheafified (negative) cyclic homology as truncating invariant}\label{appendix: cdh sheaffified HH}\label{appendix: cdh sheaffified negative cyclic homology}

For this section, let $k = \Q$ or a finite field extension thereof. It is an important consequence of \cite{KST16, LT19, CHSW08} that on schemes, $K(-)$ is given by gluing the homotopy-invariant contribution of $KH(-)$ with $HC^{-}(-/k)$ along their maps to the cdh-sheaffification $\LHCn(-/k)$; this works more generally in any characteristic using $TC(-)$. See \cite[Diagrams (3.9) and (4.1)]{EM23} for an overview. Consequently, $\LHCn(-/k)$ extends to a truncating invariant on $\Cat^{\perf}_k$ as in \cite[Lemma 5.7]{EM23}; also see \cite[\S 8.1]{EKS25}. 

This not only yields a clear extension to the equivariant setup, but further allows arguments based on semi-orthogonal decompositions (whose pieces do not a priori need to come from geometry). Such arguments are useful in practice, so we record the details. The same trick works for cyclic homology; we include this as well.

\begin{construction}\label{construction: cdh sheaffified localizing invariants}
Define the localizing invariant $\LHCn(-/ k): \Cat_{k}^{\perf} \to \Sp$ by the pushout
\begin{equation*}
    \begin{tikzcd}
        K(-) \arrow[r] \arrow[d] & HC^{-}(-/ k) \arrow[d] \\
        KH(-) \arrow[r] & \LHCn(-/ k).
    \end{tikzcd}
\end{equation*}
\end{construction}

We use the standard conventions on (co)homological shifts. Namely, the shift $[1]$ corresponds to the suspension $\Sigma$; it increases homological degree (and decreases cohomological degree) by $1$.

\begin{construction}
We define $\LHC(-/k): \Cat_{k}^{\perf} \to \Sp$ by the $[-2]$-shift of the bottom fiber sequence in
\begin{equation*}
    \begin{tikzcd}
        HC^{-}(-/ k) \arrow[r] \arrow[d] & HP(-/ k) \arrow[r] \arrow[d] & HC(-/ k)[2] \arrow[d] \\
        \LHCn(-/ k) \arrow[r] & HP(-/ k) \arrow[r] & \LHC(-/ k)[2].
    \end{tikzcd}
\end{equation*}
Here, the map $ \LHCn(-/ k) \to HP(-/ k)$ comes from the defining pushout for $\LHCn(-/ k)$ by noting that that the composition $K(-) \to HC^{-}(-/ k) \to HP(-/ k)$ factors through $K(-) \to KH(-)$ since $KH(-)$ is the initial $\A^1$-invariant localizing invariant on $\Cat_k^{\perf}$ by \cite[Example 3.2.6.(2)]{EKS25} and $HP(-/ k)$ is $\A^1$-homotopy invariant on $\Cat_k^{\perf}$ by the discussion \cite[Example 3.2.7.(3)]{EKS25}.
\end{construction}

\begin{remark}
It is an interesting question whether one can similarly define $\LHH(-/k)$ as a truncating invariant of small stable $k$-linear categories. 
\end{remark}

\begin{observation}\label{observation: cdh sheaffified versions are truncating}
The above defined $\LHCn(-/ k)$ and $\LHC(-/ k)$ are truncating invariants on $\Cat_k^{\perf}$. 
\end{observation}
\begin{proof}
First note that both $KH(-)$ and $HP(-/ k)$ are truncating invariants on $\Cat_{k}^{\perf}$ by \cite[Proposition 3.14 and Corollary 3.11]{LT19} and \cite[proof of Lemma 5.7]{EM23}. Furthermore, looking at the first pushout diagram, the fiber $K^{\inf}(-) = \fib(K(-) \to HC^{-}(-/ k)) \simeq \fib(KH(-) \to \LHCn(-/ k))$ is a truncating invariant as well. From the fiber sequence
$K^{\inf}(-) \to KH(-) \to \LHCn(-/ k)$ we deduce that $\LHC^{-}(-/ k)$ is truncating. 
Since $HP(-/ k)$ is truncating, the shifted cofiber $\LHC(-/ k)$ of $\LHCn(-/ k) \to HP(-/ k)$ is also truncating.
\end{proof}

\begin{observation}\label{observation: cdh sheaffified versions are the same on regular schemes}
If $\eX = X/G$ is a global quotient of a regular scheme $X$ by a nice group $G$, the natural maps induce equivalences
\begin{align*}
HC^{-}(\eX / k) &\xrightarrow{\simeq} \LHCn(\eX / k),\\
\HC(\eX / k) &\xrightarrow{\simeq} \LHC(\eX / k).   
\end{align*}
\end{observation}
\begin{proof}
If $\eX = X/G$ is a global quotient of a regular scheme by a nice group, the map $K(\eX) \to KH(\eX)$ is an equivalence by \cite[Theorem 1.3, (1) and (4)]{Hoy21}; the result then propagates through the defining pushouts and fiber sequences.    
\end{proof}

On schemes, the above defined invariants $\LHCn(-/k)$ and $\LHC(-/k)$ recover the usual cdh-sheaffifications of $\HC^{-}(-/k)$ and $\HC(-/k)$. In fact, this works equivariantly with respect to nice groups.

\begin{observation}\label{observation: lifting restrict to cdh sheaffification}
Let $G$ be a nice group. When restricted to invariants of $\Sch^G_{k}$, the above defined truncating invariants $(\LHCn)^G(-/ k)$ and $\LHC^G(-/ k)$ are given by the actual cdh sheafifications of $(HC^{-})^G(-/ k)$ and $HC^G(-/ k)$.
\end{observation}
\begin{proof}
Since the newly defined invariants $\LHCn(-/ k)$, $\LHC(-/ k)$ are truncating by Observation \ref{observation: cdh sheaffified versions are truncating}, they in particular satisfy cdh descent on $\Sch^G_k$ by \cite{LT19, ES21}, hence they get natural maps from the actual cdh-sheaffifications of $HC^{-}(-/ k)$ and $HC(-/ k)$ on $\Sch^G_k$. These maps are equivalences on regular schemes by Observation \ref{observation: cdh sheaffified versions are the same on regular schemes}. 
Since we are working over a field of characteristic zero (a finite field extension of $\Q$), we conclude via cdh descent from the existence of $G$-equivariant resolutions of singularities -- see \cite[\S 6.1]{EKS25} for a brief overview; this argument already appears in \cite[Corollary 1.4]{Hoy21}.
\end{proof}
\section{Variant for actions of linearly reductive groups}\label{appendix: actions of linearly reductive G}
Under further mild technical assumptions, our Theorems \ref{theorem: vanishing of truncating invariants}, \ref{theorem: a degree zero version}, \ref{theorem: computation of truncating invariants} and their consequences hold with respect to any linearly reductive group $G$ over $k$. As explained in Remark \ref{remark: conditionally G lineraly reductive}: this is irrelevant if $\chara k = p$, but considerably extends the applicability of our results to the case of any reductive group if $\chara k = 0$. For instance, this applies to the action of $GL_n$ on affine Schubert varieties, which are simple by Lemma \ref{example: affine schubert varieties are simple}. 

The aim of this section is to outline this variant, fully relying on the methods developed in \cite{LS25, KR21, Kha18, BKRS22}. Since the arguments need extra technical care, we decided to leave them as an appendix. Since the only extra content appears when $\chara k = 0$, we will assume this throughout.

\begin{setup}
Throughout this section, let $G$ be a linearly reductive group over a field $k$ of $\chara k = 0$.

For simplicity, we will state our results for the subcategory of quasi-projective $G$-equivariant varieties $\Sch^{G, \qp}_k$ (but see Recollections \ref{recollection: connective perfect generation}, \ref{recollection: linearly scalloped stacks} for relevant weaker conditions). We denote by $\eB^{G, \qp}_k := \eB^G_k \cap \Sch^{G, \qp}_k$ and $\eC^{G, \qp}_k := \eC^G_k \cap \Sch^{G, \qp}_k$ the classes of quasi-projective $G$-equivariant simple varieties over $k$.
\end{setup}

\begin{recollection}\label{recollection: connective perfect generation}
Firstly, let us recall the notion of stacks satisfying \textit{connective perfect generation} from \cite[Definition 6.21]{LS25}. In particular, suppose $X$ is a $G$-equivariant qcqs derived scheme over $k$ such that $X/G$ satisfies the resolution property \cite[Definition 6.15]{LS25}; then $X/G$ satisfies connective perfect generation. Indeed, it is a quasi-geometric stack \cite[Definition 6.1]{LS25} with affine stabilizers over a field of characteristic zero by our setup, hence of finite cohomological dimension \cite[Definition 6.10.(1), Example 6.13]{LS25}, so we conclude by \cite[Example 6.19]{LS25}. Let us further remark that $X/G$ satisfies perfect generation \cite[Definition 6.10.(2), Example 6.19]{LS25}.

Specifically, for any equivariant quasi-projective variety $X \in \Sch^{G, \qp}_k$, the quotient stack $X/G$  satisfies connective perfect generation (as in Example \cite[Example 6.22.(2)]{LS25}). Also see \cite[end of Example 6.22]{LS25} for the expected generality.

Furthermore, the resolution property for a quasi-geometric derived stack $\eX$ with affine stabilizers only depends on its classical truncation $\eX_{\cl}$. Specifically, any derived enhancement of a quotient stack $X/G$ for some $X \in \Sch^{G, \qp}_k$ as above satisfies connective perfect generation. 
%
%
%
(In this way, connective perfect generation holds for derived projectivizations \cite{Jia22a, Jia22b, Jia23} of complexes over such $X/G$, since their underlying classical truncations are again of the form $Y/G$ for some $Y \in \Sch^{G, \qp}_k$.)
\end{recollection}



\begin{proposition}\label{proposition: nilinvariance for linearly reductive G}
Let $E(-)$ be a truncating $k$-linear localizing invariant and let $\eX$ be a noetherian quasi-geometric stack satisfying connective perfect generation. Then $E(\eX) \simeq E(\eX_{\cl})$.
\end{proposition}
\begin{proof}
This is \cite[Theorem 6.35]{LS25}.
\end{proof}

\begin{recollection}\label{recollection: linearly scalloped stacks}
Further recall the notion of \textit{linearly scalloped stacks} from \cite[Definition A.1]{KR21}. Again, a quotient $X/G$ of a quasi-projective variety $X \in \Sch^{G, \qp}_k$ by a linearly reductive group $G$ falls into this class \cite[Example A.4]{KR21}.
\end{recollection}

\begin{proposition}\label{proposition: cdh for linearly reductive G}
Let $E(-)$ be a truncating $k$-linear localizing invariant. Then it satisfies cdh descend on $\Sch^{G, \qp}_k$.
\end{proposition}
\begin{proof}
We appeal to the cdh descend criterion \cite[Theorem 5.6]{Kha18}, which is applicable in the current stacky setup by \cite[Remark 5.11.(iii) and \S 5.3.4]{Kha18}. The value of $E$ on the empty scheme is trivial, giving (i). It moreover satisfied Nisnevich descent in the generality of stacks satisfying perfect generation by \cite[Corollary 4.1.2]{BKRS22}, giving (ii). The closed descent condition (iii) follows by the finite case of \cite[Theorem C]{BKRS22}, valid in the generality of linearly scalloped stacks by \cite[Remark 0.0.9]{BKRS22}; note that closed squares are finite abstract blowup squares by design \cite[Example 5.4]{Kha18}. The blow-up condition (iv) follows by combining \cite[Theorem A]{Kha18} with the derived nilinvariance from Proposition \ref{proposition: nilinvariance for linearly reductive G}, analogously to \cite[Example 5.8]{Kha18}; the derived blowups in question will satisfy connective perfect generation by the analysis from Recollection \ref{recollection: connective perfect generation}.
\end{proof}

\begin{remark}
In the special case of $E=KH$, Propositions \ref{proposition: nilinvariance for linearly reductive G}, \ref{proposition: cdh for linearly reductive G} for linearly scalloped stacks are fully covered by \cite[Theorem F, Theorem G, \S A.8.1(b)]{KR21}. In their treatment, the arguments covering derived nilinvariance are more specific to $KH$.
\end{remark}

\begin{remark}\label{remark: any affine group}
We may further bootstrap to obtain the above for any affine group scheme $G$ of finite type over $k$ of $\chara k = 0$. Indeed, such $G$ embeds into some $GL_n$ for $n$ sufficiently large. We may then formally switch from $G$-varieties to $GL_n$-varieties through
$\Sch^{G, \qp}_k \to \Sch^{GL_n, \qp}_k$, $X \mapsto X\times^G GL_n$ 
and then apply the above results with respect to the linearly reductive group $GL_n$.
\end{remark}

With the above facts in hand, we can now state the variant of our results for actions of any linearly reductive $G$.
\begin{theorem}
Let $G$ be a linearly reductive group over $k$. Then:
\begin{itemize}
    \item Theorem \ref{theorem: vanishing of truncating invariants} holds for the class $\eC^{G, \qp}_k$,
    \item Theorems \ref{theorem: a degree zero version}, \ref{theorem: computation of truncating invariants} holds for the class $\eB^{G, \qp}_k$.
\end{itemize}
\end{theorem}
\begin{proof}
We need to prove stability under the closure properties defining simple varieties. The proofs of \eqref{closure property: the point} and \eqref{closure property: projective bundles} work as before.

For \eqref{closure property: stratified projective bundles}, consider $X$, $Y= \P_X(\eE_{\bullet})$, $Y_{\cl} = \P_X(\eF)$. We have $E^G(Y) \simeq 0 \implies E^G(X) \simeq 0$ by semi-orthogonal decompositions as before. Moreover, as $E(-)$ is truncating, we have $E^G(Y_{\cl}) \simeq E^G(Y)$ by Proposition \ref{proposition: nilinvariance for linearly reductive G}, which is applicable since all derived stacks in question satisfy connective perfect generation by Recollection \ref{recollection: connective perfect generation}. This proves \eqref{closure property: stratified projective bundles part a}, while \eqref{closure property: stratified projective bundles part b} follows by Lemma \ref{lemma: ommiting flag schemes}.

For \eqref{closure property: 3-out-of-4 for split abstract blowups} and \eqref{closure property: 3-out-of-4 for abstract blowups} we need to see that $E$ sends $G$-equivariant abstract blowup squares to homotopy fiber squares; this is covered by Proposition \ref{proposition: cdh for linearly reductive G}
\end{proof}

Through this variant, the results of \S \ref{section: rational goodwillie--jones trace in degree zero} and \S \ref{section: homotopy invariant K-theory and equivariant formality} work for actions of linearly reductive groups on quasi-projective varieties. For example, this applies to affine Schubert varieties from Lemma \ref{example: affine schubert varieties are simple} with respect to $G=GL_n$ if $\chara k = 0$.

\printbibliography

\bigskip
\bigskip
\bigskip
\bigskip

\noindent Jakub Löwit, \newline
Institute of Science and Technology Austria (ISTA), \newline
Am Campus 1, \newline 
3400 Klosterneuburg, \newline
Austria \newline
\texttt{jakub.loewit@ist.ac.at}

\end{document}